\documentclass[a4paper,smallmargins,11pt]{article}
\pdfoutput=1

\usepackage{presets_philip}

\author{Philip Boeken\thanks{Department of Mathematics, VU Amsterdam, \url{p.a.boeken@vu.nl}} \and Patrick Forr\'e\thanks{Korteweg-de Vries Institute for Mathematics, University of Amsterdam} \and Joris M.\ Mooij\footnotemark[2]}
\date{\today}

\title{Are Bayesian networks typically faithful?}

\begin{document}
\maketitle

\begin{abstract}
	Faithfulness is a common assumption in causal inference, often motivated by the fact that the faithful parameters of linear Gaussian and discrete Bayesian networks are typical, and the folklore belief that this should also hold for other classes of Bayesian networks. We address this open question by showing that among all Bayesian networks over a given DAG, the faithful Bayesian networks are indeed `typical': they constitute a dense, open set with respect to the total variation metric. This does not directly imply that faithfulness is typical in restricted classes of Bayesian networks that are often considered in statistical applications. To this end we consider the class of Bayesian networks parametrised by conditional exponential families, for which we show that under regularity conditions, the faithful parameters constitute a dense and open set, the unfaithful parameters have Lebesgue measure zero, and the induced faithful distributions are open and dense in the weak topology. This  extends the existing results for linear Gaussian and discrete Bayesian networks. We also show for nonparametric classes of Bayesian networks with uniformly equicontinuous and uniformly bounded conditional densities that the faithful Bayesian networks are open and dense in the weak topology. All these results also hold for Bayesian networks with latent variables, if faithfulness is only required to hold with respect to the latent projection. Finally, for the considered conditional exponential family parametrisations and nonparametric conditional density models, the topological properties of conditional independence imply the existence of a consistent conditional independence test. Together with the topological properties of faithfulness, this implies that sound constraint-based causal discovery algorithms like \emph{PC} and \emph{FCI} are consistent on an open and dense -- and hence `typical' -- set of Bayesian networks.
\end{abstract}

\section{Introduction}
Given a Bayesian network over a DAG $G$ with variables $V$ and a finite sample from its distribution $\PP(X_V)$, the task of \emph{causal discovery} algorithms is to infer the graph $G$ from the data. \emph{Constraint-based} causal discovery methods do so by testing for conditional independencies $X_A \Indep_\PP X_B \given X_C$ for various choices of $A, B, C \subseteq V$, and use this information to reconstruct $G$, up to certain equivalences. A core assumption of many constraint-based causal discovery algorithms is that a correctly inferred set of conditional independencies in $\PP(X_V)$ characterises the corresponding set of $d$-separations in $G$: for all subsets of vertices $A, B, C \subseteq V$ we have
\begin{equation*}
	A \perp_{G}^d B \given C \iff X_A\Indep_\PP X_B\given X_C.
\end{equation*}
Bayesian networks for which this condition holds are called \emph{faithful} \citep{pearl1988probabilistic,spirtes1993causation}. The implication from left to right holds for all Bayesian networks, and is called the \emph{global Markov property} \citep{verma1990causal}. The implication from right to left does not always hold: there exist Bayesian networks which have conditional independencies that are not due to a corresponding $d$-separation in the graph --- instead, they might be due to cancelling paths, deterministic variables, or deterministic relations (see Example \ref{ex:unfaithful} below).

In absence of any knowledge of the graph $G$, faithfulness is an untestable assumption \citep{zhang2008detection}. In practice, this assumption is often motivated by theoretical results that for certain parametric models, the faithful distributions are `typical'. For a given DAG $G$, \cite{spirtes1993causation} and \cite{meek1995strong} consider specific parametrisations $\Theta_\Ncal$ and $\Theta_\Dcal$ of linear Gaussian and discrete Bayesian networks respectively (which are subsets of $\RR^d$ for appropriate $d\in\NN$, see Examples \ref{ex:exp_gauss} and \ref{ex:exp_discrete} below in Section \ref{sec:exp}) and show that drawing the parameter values at random will give a faithful Bayesian network with probability one:
\begin{theorem}[\citealp{spirtes1993causation}]\label{thm:spirtes}
	With respect to Lebesgue measure over $\Theta_\Ncal$, the subset of parameter values whose distribution is unfaithful to $G$ is measure-zero.
\end{theorem}
\begin{theorem}[\citealp{meek1995strong}]\label{thm:meek}
	With respect to Lebesgue measure over $\Theta_\Dcal$, the subset of parameter values whose distribution is unfaithful to $G$ is measure-zero.
\end{theorem}

To our knowledge, no such results are available for other parametric or nonparametric classes of distributions. In this work we prove such a result: without restriction to any parametric or nonparametric class of distributions, the faithful distributions are typical. As there is no canonical analogue of the Lebesgue measure for the nonparametric space of Bayesian networks, we don't consider the measure-theoretic notion of typicality, but instead consider a topological notion. Our most general nonparametric result, Theorem \ref{thm:bn_dist_nowhere_dense}, is as follows:

\begin{quote}\normalsize
	\textbf{Among all distributions that are Markov with respect to a given DAG, the faithful distributions constitute a dense, open set.}
\end{quote}
As a consequence, the set of faithful distributions is non-empty, and unfaithful distributions are \emph{nowhere dense} (defined below) and are thus `atypical'. This topological property is with respect to the total variation metric on the joint distribution $\PP(X_V)$ over all vertices $V$ of the Bayesian network. This result holds for any choice of \emph{standard Borel} outcome spaces; it holds in particular for continuous variables $X_V\in \RR^{|V|}$, discrete variables $X_V\in \ZZ^{|V|}$, and mixed data. Formally, a Bayesian network is a DAG $G$ with for every vertex $v$ a Markov kernel $\PP(X_v\given X_{\pa(v)})$, where $\pa(v)$ denote the parents of $v$ in $G$. The above result is about the observational distributions, but not about the Bayesian networks themselves, which are tuples of Markov kernels (one for each $v\in V$). To this end, we introduce a metric $d_{TV}^\circ$ on the space of Bayesian networks --- corresponding to total variation convergence of the Markov kernels, uniformly in the conditioning variables --- for which we show in Theorem \ref{thm:bn_params_nowhere_dense} that the faithful Bayesian networks are open and dense.

These claims do not automatically generalise to other topologies than the total variation topology or the topology induced by $d_{TV}^\circ$: open sets are maintained under refinements, and dense sets are maintained under coarsening, but nothing can be said in general about open and dense sets. Besides the total variation topology, the weak topology is of particular interest since it is tightly connected to statistical testability \citep{dembo1994topological,ermakov2017consistent,genin2017topology,boeken2026topological}. Since faithfulness is not open in the weak topology, some regularity conditions are necessary to obtain the analogue of the above result in the weak topology.

In practice, one often imposes regularity assumptions on the distribution to facilitate statistical inference. To this end, we consider two subclasses of Bayesian networks. First, we consider the class of Bayesian networks parametrised by conditional exponential families. Under regularity conditions, we obtain in Theorem \ref{thm:exp_params_nowhere_dense} that if there exists a faithful parameter, then the faithful parameters constitute a dense and open set, and the unfaithful parameters have Lebesgue measure zero. We further show in Theorem \ref{thm:exp_dist_nowhere_dense} that the induced set of faithful observational distributions is open and dense with respect to the weak topology.
Second, we consider nonparametric models of Bayesian networks with uniformly equicontinuous and uniformly bounded conditional densities. In this class we also obtain that if there exists a faithful model, then the faithful models constitute a dense and open set with respect to $d_{TV}^\circ$ (Theorem \ref{thm:equicont_params_lebesgue_nowhere_dense}). Since for this model class the weak topology and the total variation topology coincide, convergence in this metric corresponds to weak convergence of the Markov kernels, uniformly in the conditioning variable. We also show that the induced set of faithful observational distributions is open and dense with respect to the weak topology (Theorem \ref{thm:equicont_dist_nowhere_dense}).

To relate these results to constraint-based causal discovery, we show that for the subclasses of Bayesian networks with conditional exponential families or nonparametric conditional densities, it follows from results by \cite{genin2017topology} and \cite{lauritzen2024total} that there exists a consistent conditional independence test, and hence that any sound constraint-based causal discovery algorithm using such a test is consistent on an open and dense set of Bayesian networks (Theorem \ref{thm:consistent_discovery}). This holds for settings where the samples spaces are separable complete metric spaces.

\vspace{1em}
There exist multiple mathematical notions of `atypicality'. Given a set $M$, `small' subsets of $M$ are characterised by so-called $\sigma$-ideals: collections of subsets of $M$ containing $\emptyset$, that are closed under taking subsets and countable unions. The family of Lebesgue measure zero sets is a $\sigma$-ideal, and so is the family of meager sets:

\begin{definition}
	A set $I \subseteq M$ is \emph{dense} in another set $U\subseteq M$ if every point in $U$ is in $I$ or is a limit point of $I$. The set $I$ is \emph{nowhere dense} if there is no non-empty open subset of $M$ in which $I$ is dense, and it is \emph{meager} if it is a countable union of nowhere dense sets.
\end{definition}

For example, the set of integers $\ZZ$ is nowhere dense in $\RR$, and the rationals $\QQ$ are meager in $\RR$. The boundary of every open or closed set is nowhere dense, and subsets of nowhere dense sets are nowhere dense. Complements of dense sets are not necessarily nowhere dense or meager, but complements of \emph{dense and open} sets are nowhere dense. Comeager sets (complements of meager sets) are commonly referred to as \emph{typical} \citep{kechris2012classical}. Our results imply that unfaithful distributions and parameters are nowhere dense, which is an even stronger notion of atypicality.

In causality, the $\sigma$-ideal of meager sets is considered by \cite{ibeling2021topological}, who show that discrete causal models for which \emph{Pearl's Causal Hierarchy} collapses\footnote{A structural causal model `collapses' when all counterfactual  (interventional) queries are identifiable from interventional (observational) distributions.} are meager, which is a topological analogue of a Lebesgue measure-zero result from \cite{bareinboim2022pearls}. \cite{lin2020learning} prove open- and denseness of faithful parameters of discrete Bayesian networks.


\vspace{1em}
The outline of this paper is as follows. In Section \ref{sec:technical} we provide some technical prerequisites about Bayesian networks, the total variation metric, the bounded-Lipschitz metric and the weak topology. In Section \ref{sec:bn} we prove for unconstrained Bayesian networks that faithful distributions are dense and open in the total variation metric, and that the faithful Bayesian networks are open and dense in a newly introduced metric $d_{TV}^\circ$. In Section \ref{sec:exp} we focus on conditional exponential family parametrisations of Bayesian networks, where we show that the faithful parameters are open and dense in the Euclidean parameter space --- generalising the results of Spirtes et al.\ and Meek for linear Gaussian and discrete Bayesian networks --- and that the induced faithful distributions are open and dense in the weak topology. In Section \ref{sec:equicont} we prove for classes of distributions with uniformly equicontinuous and uniformly bounded densities that faithful Bayesian networks are open and dense with respect to $d_{TV}^\circ$, and that faithful distributions are open and dense with respect to the weak topology. In Section \ref{sec:bayesnet_latent} we extend our results to Bayesian networks with latent variables. Finally, in Section \ref{sec:discussion} we discuss the relative value of these results with their various notions of typicality. We explicitly show how topological properties of conditional independence and faithfulness imply the existence of consistent conditional independence tests, and with that a constraint-based causal discovery algorithm that is consistent on an open and dense set of Bayesian networks.

\section{Technical prerequisites}\label{sec:technical}
A \emph{directed acyclic graph} (DAG) is a tuple $G = (V, E)$ with $V$ a finite set of vertices and $E\subset V\times V$ a set of directed edges such that there are no directed cycles. Given such a finite index set $V$, let $\Xcal_V = \prod_{v\in V}\Xcal_v$ be a product of separable complete metric spaces, each containing at least two points and equipped with the Borel $\sigma$-algebra $\Bcal(\Xcal_v)$ (which are \emph{standard Borel spaces}), and let $\Pcal(\Xcal_V)$ be the set of probability measures on $\Xcal_V$. Random variables will be denoted with $X_V$, and their values with $x_V$. For $A, B \subseteq V$, a \emph{Markov kernel} $\PP(X_B \given X_A)$ is a measurable map $\Xcal_A \to \Pcal(\Xcal_B)$, where $\Pcal(\Xcal_B)$ is equipped with the smallest $\sigma$-algebra that makes for all $D\in \Bcal(\Xcal_B)$ the evaluation map $\ev_D : \Pcal(\Xcal_B) \to [0,1], \PP \mapsto \PP(X_B \in D)$ measurable. For Markov kernels $\PP(X_A \given X_B), \PP(X_B\given X_C)$, their \emph{product} is defined as the Markov kernel
\begin{equation*}
	\PP(X_A \given X_B) \otimes \PP(X_B\given X_C) : \Xcal_C \to \Pcal(\Xcal_{A\cup B}), ~~~ x_C\mapsto \left(D \mapsto \int_{D}\diff\PP(x_A \given x_B)\diff\PP(x_B \given x_C)\right)
\end{equation*}
where $D\in \Bcal(\Xcal_{A\cup B})$. Since $\Xcal_V$ is standard Borel, there exists for any joint distribution $\PP(X_A, X_B)$ (where $A, B \subseteq V$) a Markov kernel (often referred to as \emph{conditional distribution}) $\PP(X_B \given X_A)$ such that $\PP(X_A, X_B) = \PP(X_B \given X_A) \otimes \PP(X_A)$ (\citealp{bogachev2007measurea}, Corollary 10.4.15). Given distribution $\PP \in \Pcal(\Xcal_V)$ and sets $A, B, C \subseteq V$, we say that $X_A$ is \emph{conditionally independent} of $X_B$ given $X_C$, written $X_A \Indep_\PP X_B \given X_C$, if for all $E_A \in \Bcal(\Xcal_A)$ and $E_B \in \Bcal(\Xcal_B)$ we have $\PP(X_A\in E_A, X_B\in E_B \given X_C) = \PP(X_A\in E_A \given X_C)\PP(X_B\in E_B \given X_C)$ a.s.
If $\PP(X_A, X_B, X_C)$ has a density $p(x_A, x_B, x_C)$, we have $X_A\Indep_\PP X_B\given X_C$ if and only if $p(x_A, x_B\given x_C) = p(x_A\given x_C)p(x_B\given x_C)$ a.e. If the density is continuous, this is equivalent to the factorisation holding everywhere in the support of $\PP(X_A, X_B, X_C)$.

For any vertex $v\in V$, its parents in $G$ are given by the set $\pa(v) := \{w \in V : w\to v \in E\}$, and its ancestors are given by $\{v\} \cup \{w \in V : \exists\text{ a path } (w\to ... \to v) \in G\}$.
A \emph{Bayesian network}\footnote{In statistical literature, Bayesian networks are often defined as joint distributions $\PP(X_V)$, from which one must deduce the conditional distributions, which are not uniquely defined. For \emph{causal} modelling, it is more suitable to model the Bayesian network as a set of Markov kernels, which uniquely specify the effects of interventions. Our results about the typicality of faithfulness cover both viewpoints: they are shown in the space of observational distributions, and in the space of Bayesian networks.} consists of a DAG $G$ and a tuple of Markov kernels $(\PP(X_v\given X_{\pa(v)}))_{v\in V}$. The joint distribution $\PP(X_V) = \bigotimes_{v\in V} \PP(X_v \given X_{\pa(v)})$ is referred to as the \emph{observational distribution}. Given DAG $G$ with path $\pi=a\sus...\sus b$, a \emph{collider} is a vertex $v$ with $...\to v \ot...$ in $\pi$, where $\sus$ is a placeholder for either $\to$ or $\ot$. For sets of vertices $A, B, C\subseteq V$ we say that $A$ and $B$ are \emph{$d$-separated} given $C$, written $A\perp^d_G B \given C$, if for every path $\pi = a \sus ... \sus b$ between $a\in A$ and $b \in B$, there is a collider on $\pi$ that is not an ancestor of $C$, or there is a non-collider on $\pi$ in $C$. We adopt the convention that for $a = b$ the trivial path consisting of the single vertex $a$ counts as a path, which is blocked precisely when $a \in C$. The sets $A$ and $B$ are \emph{$d$-connected} given $C$ if they are not $d$-separated, written $A\nperp^d_G B \given C$.
\begin{definition}
	Given a DAG $G$ and distribution $\PP$, we say that $\PP$ is \emph{Markov with respect to $G$} if for all $A, B, C\subseteq V$ we have
	\begin{equation}\label{eqn:mp}
		A\perp^d_G B \given C \implies X_A \Indep_\PP X_B \given X_C.
	\end{equation}
	For a pair $(G, \PP)$, this is also referred to as the \emph{global Markov property}.
\end{definition}

\begin{theorem}[\citealp{verma1990causal}]
	The global Markov property holds for all Bayesian networks.
\end{theorem}
For a general Bayesian network, the set of conditional independencies in its observational distribution $\PP$ does not characterise the set of $d$-separations in $G$: we might have a $d$-connection $A\nperp^d_G B \given C$ but still have a conditional independence $X_A \Indep_\PP X_B \given X_C$. A Bayesian network is called \emph{faithful} if these cases are excluded:
\begin{definition}
	Given a DAG $G$ and distribution $\PP$, we say that $\PP$ is \emph{faithful with respect to $G$} if for all $A, B, C\subseteq V$ we have
	\begin{equation*}
		A\nperp^d_G B \given C \implies X_A \nIndep_\PP X_B \given X_C.
	\end{equation*}
	A Bayesian network is faithful if its observational distribution is faithful with respect to its graph.
\end{definition}

\begin{example}\label{ex:unfaithful}
	The following Bayesian networks are unfaithful. The corresponding graphs $G^a, G^b$ and $G^c$ are depicted in Figure \ref{fig:unfaithful_example}.
	\begin{enumerate}[label=\alph*)]
		\item Cancelling paths: let $\PP(X_A)$ be any distribution and let $\PP(X_B \given X_A) = \Ncal(\beta_{AB} X_A, \sigma_B^2)$ and $\PP(X_C \given X_A, X_B) = \Ncal(\beta_{AC} X_A + \beta_{BC} X_B, \sigma_C^2)$ for given variances $\sigma_A^2, \sigma_B^2, \sigma_C^2 >0$ and coefficients $\beta_{AC}, \beta_{AB}, \beta_{BC} \in \RR$ with $\beta_{AC} = -\beta_{AB}\beta_{BC}$. Then $A\nperp^d_{G^a} C$ and $X_A\Indep X_C$.\footnote{A realistic example is when opening up a window $(A)$ signals the thermostat to turn up the heating $(B)$, so the inflow of cold air is perfectly offset by the heating, causing a net zero effect on room temperature $(C)$.}
		\item Deterministic variables: let $\PP(X_A\given X_B)$ and $\PP(X_C\given X_B)$ be Markov kernels and let $\PP(X_B)= \delta_{x_B}$ for some $x_B\in \Xcal_B$, so $X_B$ deterministically has the value $x_B$. Then we have $A\nperp^d_{G^b} C$ and $X_A\Indep X_C$.
		\item Deterministic relations: let $\PP(X_A\given X_D)$ and $\PP(X_C\given X_D)$ be Markov kernels and $\PP(X_D)$ any distribution and let $\PP(X_B\given X_D)  = \delta_{X_D}$, so we deterministically set $X_B = X_D$. Then we have $A\nperp^d_{G^c} C \given B$ and $X_A\Indep X_C\given X_B$.\footnote{For Bayesian networks with known deterministic variables or relations, \cite{geiger1990identifying} introduced the \emph{$D$-separation} criterion. This takes into account part of the determinism to deduce additional conditional independencies that are not implied by the $d$-separation criterion.}
	\end{enumerate}
\end{example}
\begin{figure}[!h]
	\centering
	\begin{subfigure}{.3\linewidth}
		\centering
		\begin{tikzpicture}
			\node[var] (A) at (-1,0) {$A$};
			\node[var] (B) at (0,.75) {$B$};
			\node[var] (C) at (1,0) {$C$};
			\draw[arr] (A) to (B);
			\draw[arr] (B) to (C);
			\draw[arr] (A) to (C);
		\end{tikzpicture}
		\caption{DAG $G^a$}
		\label{fig:unfaithful_example:1}
	\end{subfigure}
	\begin{subfigure}{.3\linewidth}
		\centering
		\begin{tikzpicture}
			\node[var] (A) at (-1,0) {$A$};
			\node[var] (B) at (0,.75) {$B$};
			\node[var] (C) at (1,0) {$C$};
			\draw[arr] (B) to (A);
			\draw[arr] (B) to (C);
		\end{tikzpicture}
		\caption{DAG $G^b$}
		\label{fig:unfaithful_example:2}
	\end{subfigure}
	\begin{subfigure}{.3\linewidth}
		\centering
		\begin{tikzpicture}
			\node[var] (A) at (-1,0) {$A$};
			\node[var] (D) at (0,.75) {$D$};
			\node[var] (B) at (0,-.5) {$B$};
			\node[var] (C) at (1,0) {$C$};
			\draw[arr] (D) to (A);
			\draw[arr] (D) to (C);
			\draw[arr] (D) to (B);
		\end{tikzpicture}
		\caption{DAG $G^c$}
		\label{fig:unfaithful_example:3}
	\end{subfigure}
	\caption{DAGs of the Bayesian networks that are given in Example \ref{ex:unfaithful}.}
	\label{fig:unfaithful_example}
\end{figure}

An important step in our proof of the typicality of faithful distributions, is that conditional independence is a topologically closed property, which means that it is preserved by taking limits. Whether this holds depends on the particular choice of the topology on $\Pcal(\Xcal_V)$. A well-known topology is the one related to weak convergence: given probability measures $\PP, \PP_1, \PP_2, ... \in \Pcal(\Xcal_V)$ we say that \emph{$\PP_n$ converges weakly to $\PP$} (also known as \emph{convergence in distribution}) if $\EE_{\PP_n}[f] \to \EE_\PP[f]$ for all continuous functions $f:\Xcal_V\to [-1,1]$, denoted by $\PP_n \wto \PP$. This topology is metrised by the \emph{bounded-Lipschitz} metric $d_{BL}(\PP, \QQ) := \sup_f \left|\int f\diff\PP - \int f\diff \QQ\right|$, where the supremum is taken over all functions $f: \Xcal_V \to [-1,1]$ with $\Lip(f)\leq 1$, where $\Lip(f)$ denotes the Lipschitz constant of $f$ (\citealp{bogachev2007measurea}, Theorem 8.3.2). However, weak convergence does not necessarily preserve conditional independence: for a weakly convergent sequence $\PP_n\to \PP$ with $X_A\Indep_{\PP_n} X_B\given X_C$ for all $n\in\NN$, we might have $X_A\nIndep_{\PP} X_B \given X_C$ in the limit; see e.g.\ \cite{lauritzen1996graphical}, pp. 38-39 for an example. If the sample space $\Xcal_C$ is uncountable, then the set $\{\PP : X_A\Indep_{\PP} X_B\given X_C\}$ is even dense in the weak topology on $\Pcal(\Xcal_{A\cup B\cup C})$ \citep{boeken2026topological}.

Due to this fact, we also consider the total variation topology on $\Pcal(\Xcal_V)$, induced by the \emph{total variation metric} $d_{TV}(\PP, \QQ) := \sup_{A \in \Bcal(\Xcal_V)}|\PP(A) - \QQ(A)|$. This can equivalently be written as $d_{TV}(\PP, \QQ) = \frac{1}{2}\sup_f \left|\int f\diff\PP - \int f\diff \QQ\right|$ where the supremum is taken over all measurable functions $f:\Xcal_V\to [-1,1]$, so it is generally stronger than weak convergence.
Convergence in this metric is denoted by $\PP_n \tvto \PP$.
By \cite{lauritzen2024total} we have that conditional independence is closed in total variation:
\begin{theorem}[\citealp{lauritzen2024total}]\label{thm:ci_is_closed}
	Given probability measures $\PP, \PP_1, \PP_2, ... \in \Pcal(\Xcal_V)$ such that $\PP_n \tvto \PP$, if we have $X_A\Indep_{\PP_n} X_B\given X_C$ for all $n\in\NN$, then also $X_A\Indep_\PP X_B\given X_C$.
\end{theorem}

\section{Unconstrained Bayesian networks}\label{sec:bn}
In this section, we consider the typicality of faithfulness in the unconstrained class of Bayesian networks; that is, Bayesian networks without any assumptions on the Markov kernels. Faithfulness is a property of the observational distribution of a Bayesian network, so we first consider in Section \ref{sec:bn_dists} the typicality of faithful observational distributions in the space of all distributions that are Markov with respect to the given DAG, equipped with the total variation metric. In Section \ref{sec:bn_params} we consider the typicality of faithfulness in the space of Bayesian networks (which are Markov with respect to the given DAG $G$) equipped with a newly introduced metric.

\subsection{Typicality of faithfulness in the space of observational distributions}\label{sec:bn_dists}
Given a DAG $G=(V, E)$, we consider the following sets of Markov, faithful, and unfaithful distributions relative to $G$:
\begin{align*}
	M_G & := \left\{\PP \in \Pcal(\Xcal_V) : A\perp^d_G B \given C \implies X_A \Indep_\PP X_B \given X_C \textrm{ for all $A, B, C\subseteq V$} \right\} \\
	F_G & := \left\{\PP \in M_G : A\nperp^d_G B \given C \implies X_A \nIndep_\PP X_B \given X_C \textrm{ for all $A, B, C\subseteq V$} \right\}          \\
	U_G & := M_G \setminus F_G.
\end{align*}
We will derive properties of $F_G$ and $U_G$ as subsets of the metric space $(M_G, d_{TV})$; in later sections we will add regularity conditions on $M_G, F_G$ and $U_G$, and consider other topologies. First, if we let $I_{A,B|C} = \{\PP\in \Pcal(\Xcal_V) : X_A\Indep_\PP X_B\given X_C\}$, note that we can write
\begin{align*}
	M_G = \Pcal(\Xcal_V) \cap \bigcap_{A\perp^d_G B\given C}I_{A,B|C},
	\quad\quad F_G = M_G\cap \bigcap_{A\nperp^d_G B\given C}(M_G\setminus I_{A,B|C}),
	\quad\quad U_G  = M_G\setminus F_G.
\end{align*}
From Theorem \ref{thm:ci_is_closed} it is immediate that $M_G$ is a closed subspace of $\Pcal(\Xcal_V)$, and that $F_G$ is open in $M_G$. For our most general nonparametric result, it remains to show that $F_G$ is dense.
The following result states that the set of distributions that are Markov \emph{and} have a particular conditional dependence is dense in total variation. The proof refers to technical lemmas that are provided in Section \ref{sec:cd_dense_proof}.
\begin{lemma}\label{thm:cd_dense}
	For every $\PP\in M_G$ and every $A, B, C \subseteq V$ such that $A \nperp_G^d B\given C$, there is a sequence $\PP_1, \PP_2, ... \in M_G$ such that $X_A \nIndep_{\PP_n} X_B \given X_C$ for all $n\in\NN$ and $\PP_n \tvto \PP$.
\end{lemma}
\begin{proof}
	Let $\PP\in M_G$ be given such that $X_A\Indep_{\PP} X_B \given X_C$ --- otherwise the result holds trivially. By Lemma \ref{thm:d_weak_completeness}, there exists a $\PP_1$ that is Markov and has $X_A\nIndep_{\PP_1} X_B \given X_C$. The interpolation $(\PP_\lambda)_{\lambda\in(0,1)}$ between $\PP$ and $\PP_1$ from Definition \ref{def:markov_interpolation} lies in $M_G$, and has that $X_A\nIndep_{\PP_\lambda} X_B \given X_C$ for all positive $\lambda$ below some $\lambda^* \in (0,1)$ (Lemma \ref{thm:interpolation_dependence}), which converges in total variation to $\PP$ as $\lambda \to 0$ (Lemma \ref{thm:interpolation_convergence}). One obtains a suitable sequence by setting $\PP_n := \PP_{\lambda^*/2n}$.
\end{proof}

In other words, the set $\{\PP \in M_G : X_A \nIndep_{\PP} X_B \given X_C \}$ is dense in $M_G$. As a corollary that might be of independent interest, we have that conditional dependence is dense in total variation.
\begin{corollary}
	For all $A, B, C \subseteq V$ with $A\setminus C \neq \emptyset$ and $B\setminus C\neq\emptyset$, the set $\{\PP \in \Pcal(\Xcal_V) : X_A \nIndep_{\PP} X_B \given X_C \}$ is dense in $(\Pcal(\Xcal_V), d_{TV})$.
\end{corollary}
\begin{proof}
	Let $G$ be a fully connected DAG with vertices $V$, then $M_G=\Pcal(\Xcal_V)$ and $A\nperp^d_G B\given C$, so the result follows from Lemma \ref{thm:cd_dense}.
\end{proof}

Our first result concerning the faithfulness of nonparametric Bayesian networks is as follows.
\begin{theorem}\label{thm:bn_dist_nowhere_dense}
	Given a DAG $G$, the set of faithful distributions $F_G$ is non-empty, open and dense, and the unfaithful distributions $U_G$ are nowhere dense in $(M_G, d_{TV})$.
\end{theorem}
\begin{proof}
	By Theorem \ref{thm:ci_is_closed} and Lemma \ref{thm:cd_dense} we have for any given $A, B, C \subseteq V$ with $A\nperp_G^d B \given C$ that $M_G\setminus I_{A,B|C}$ is dense and open in $M_G$. Hence, $F_G$ is a dense open set as it is a finite intersection of dense open sets. Since $M_G$ is non-empty (take for example a product of independent distributions), the dense set $F_G$ is non-empty as well, proving the existence of a faithful distribution. Finally, $U_G$ is the complement of a dense open set, hence nowhere dense.
\end{proof}

To conclude, unfaithful distributions are `atypical': there is no non-empty open set of distributions that are Markov with respect to $G$, in which any faithful distribution in this set can be approximated by unfaithful ones. This loosely says that there is no `cluster' of unfaithful distributions.

\subsubsection{Conditional dependence is dense in total variation}\label{sec:cd_dense_proof}
In this section, we fill in the details of the proof of Theorem \ref{thm:bn_dist_nowhere_dense}.
\begin{lemma}\label{thm:d_weak_completeness}
	For any DAG $G$, standard Borel space $\Xcal_V$ and subsets $A, B, C \subseteq V$ such that $A\nperp_G^d B\given C$, there exists a distribution $\PP \in M_G$ with the conditional dependence $X_A\nIndep_\PP X_B \given X_C$.
\end{lemma}
\begin{proof}
	For each $v\in V$ pick an injective $f_v: \{0,1\} \to \Xcal_v$ and note that sets $\{f_v(0)\}$ and $\{f_v(1)\}$ are measurable since $\Xcal_v$ is standard Borel. We will construct a binary distribution on the image of $f_V$ that has the required dependence.
	Note that without loss of generality we can assume that $A$ and $B$ are singletons: any $\PP(X_V)$ with $X_A\nIndep_\PP X_B \given X_C$ also has $X_{A'}\nIndep_\PP X_{B'} \given X_C$ for supersets $A\subset A'$ and $B\subset B'$. Also, the given $d$-connection implies $A, B\notin C$.
	If we have $A=B$, for all $v\in V$ set $\PP(X_v = f_v(0)) = p$ and $\PP(X_v = f_v(1))=1-p$ for some $p\in(0,1)$ and let $\PP(X_V) = \bigotimes_{v\in V}\PP(X_v)$. Then $\PP(X_V)$ is Markov and $X_A \nIndep_\PP X_B\given X_C$. If $A\neq B$, then by \cite{meek1998graphical} Lemma 3,\footnote{\cite{meek1995strong} proves this result assuming weak transitivity of binary distributions, which does not hold in general. \cite{meek1998graphical} provides a correct proof based on \emph{marginal} weak transitivity.} there exists a distribution $\tilde{\PP}$ on $\{0,1\}^{|V|}$ that is Markov with respect to $G$ and which has the conditional dependence $X_A\nIndep_{\tilde{\PP}} X_B \given X_C$, so there are $\tilde{x}_A, \tilde{x}_B \in \{0,1\}$ and $\tilde{x}_C \in \{0,1\}^{|C|}$ with $\tilde{\PP}(\tilde{x}_C) > 0$ such that $\tilde{\PP}(\tilde{x}_A, \tilde{x}_B \given \tilde{x}_C)\neq \tilde{\PP}(\tilde{x}_A \given \tilde{x}_C)\tilde{\PP}(\tilde{x}_B \given \tilde{x}_C)$. Define the pushforward $\PP(X_V) := \tilde{\PP}\circ f_V^{-1}$, which has
	\begin{multline*}
		\PP\big(X_A = f_A(\tilde{x}_A), X_B = f_B  (\tilde{x}_B) \given X_C=f_C(\tilde{x}_C)\big)
		= \tilde{\PP}\big(\tilde{x}_A, \tilde{x}_B \given \tilde{x}_C\big)                                                      \\
		\neq \tilde{\PP}\big(\tilde{x}_A \given \tilde{x}_C\big)\tilde{\PP}\big(\tilde{x}_B \given \tilde{x}_C\big)             = \PP\big(X_A = f_A(\tilde{x}_A)\given X_C=f_C(\tilde{x}_C)\big)\PP\big(X_B = f_B(\tilde{x}_B) \given X_C=f_C(\tilde{x}_C)\big)
	\end{multline*}
	so indeed $X_A\nIndep_\PP X_B \given X_C$. By a similar reasoning, for any $A, B, C \subseteq V$ a conditional independence $X_A\Indep_{\tilde{\PP}} X_B \given X_C$ implies $X_A\Indep_\PP X_B \given X_C$, and thus $\PP\in M_G$.
\end{proof}

Next, we aim to construct an interpolation of any two given $\PP_0, \PP_1 \in M_G$, within $M_G$. Naively taking a mixture of the observational distributions does not give a distribution that is Markov with respect to $G$, as is shown in the following example.
\begin{example}
	Let $(\PP_i(X_A\given X_C), \PP_i(X_B\given X_C), \PP_i(X_C))$ for $i\in\{0,1\}$ be Bayesian networks with DAG $G$ as depicted in Figure \ref{fig:bn_interpolation_1}, which both have $X_A\Indep X_B \given X_C$. A mixture of the observational distributions $\tilde{\PP}_\lambda(X_A, X_B, X_C) := (1-\lambda)\PP_0(X_A, X_B, X_C) + \lambda\PP_1(X_A, X_B, X_C)$ would correspond to the $(A\cup B\cup C)$-marginal of the Bayesian network $(\PP_\alpha(X_A\given X_C), \PP_\alpha(X_B\given X_C), \PP_\alpha(X_C), \PP(\alpha))$ with $\alpha \sim \mathrm{Bernoulli}(\lambda)$. Its graph is depicted in Figure \ref{fig:bn_interpolation_2}, from which we see that $\tilde{\PP}_\lambda$ need not be Markov with respect to $G$, as we might have $X_A\nIndep_{\tilde{\PP}_\lambda} X_B \given X_C$. Instead, taking a mixture of the Markov kernels of the Bayesian networks gives $(\PP_{\alpha_A}(X_A\given X_C), \PP_{\alpha_B}(X_B\given X_C), \PP_{\alpha_C}(X_C), \PP(\alpha_A), \PP(\alpha_B), \PP(\alpha_C))$ with $\alpha_A, \alpha_B, \alpha_C \sim \mathrm{Bernoulli}(\lambda)$ i.i.d., whose $(A\cup B\cup C)$-marginal $\PP_\lambda(X_A, X_B, X_C)$ (see Definition \ref{def:markov_interpolation} below) is Markov with respect to $G$ (see Figure \ref{fig:bn_interpolation_3}).
\end{example}
\begin{figure}[htb]
	\centering
	\begin{subfigure}{.3\linewidth}
		\centering
		\begin{tikzpicture}
			\node[var] (A) at (-1,0) {$A$};
			\node[var] (C) at (0,.75) {$C$};
			\node[var] (B) at (1,0) {$B$};
			\draw[arr] (C) to (A);
			\draw[arr] (C) to (B);
		\end{tikzpicture}
		\caption{DAG $G$}
		\label{fig:bn_interpolation_1}
	\end{subfigure}
	\begin{subfigure}{.3\linewidth}
		\centering
		\begin{tikzpicture}
			\node[var] (A) at (-1,0) {$A$};
			\node[var] (B) at (1,0) {$B$};
			\node[var] (C) at (0,.75) {$C$};
			\node[var] (alpha) at (0,2) {$\alpha$};
			\draw[arr] (C) to (A);
			\draw[arr] (C) to (B);
			\draw[arr] (alpha) to (A);
			\draw[arr] (alpha) to (B);
			\draw[arr] (alpha) to (C);
		\end{tikzpicture}
		\caption{Non-Markov mixture}
		\label{fig:bn_interpolation_2}
	\end{subfigure}
	\begin{subfigure}{.3\linewidth}
		\centering
		\begin{tikzpicture}
			\node[var] (A) at (-1,0) {$A$};
			\node[var] (B) at (1,0) {$B$};
			\node[var] (C) at (0,.75) {$C$};
			\node[var] (alphaA) at (-1.2,1.25) {$\alpha_A$};
			\node[var] (alphaC) at (0,2) {$\alpha_C$};
			\node[var] (alphaB) at (1.2,1.25) {$\alpha_B$};
			\draw[arr] (C) to (A);
			\draw[arr] (C) to (B);
			\draw[arr] (alphaA) to (A);
			\draw[arr] (alphaB) to (B);
			\draw[arr] (alphaC) to (C);
		\end{tikzpicture}
		\caption{Markov mixture}
		\label{fig:bn_interpolation_3}
	\end{subfigure}
	\label{fig:bn_interpolation}
	\caption{Graphs relating to different mixtures of Bayesian networks with graph $G$.}
\end{figure}

The issue that is detailed in the previous example is resolved in Definition \ref{def:markov_interpolation}.
\begin{definition}\label{def:markov_interpolation}
	Given a DAG $G$ and two distributions $\PP_0, \PP_1 \in M_G$, define the interpolation
	\begin{equation*}
		\PP_\lambda(X_V) := \bigotimes_{v\in V}\left((1-\lambda)\PP_0(X_v \given X_{\pa(v)}) + \lambda \PP_1(X_v \given X_{\pa(v)})\right).
	\end{equation*}
\end{definition}
Note that the Markov kernels are only almost surely unique, and since the mixture kernels are evaluated also outside the supports of $\PP_0$ and $\PP_1$, the interpolation $\PP_\lambda$ genuinely depends on the choice of their versions. We therefore read Definition \ref{def:markov_interpolation} as being relative to an arbitrary fixed choice of the Markov kernels; all subsequent results hold for every such choice.
When $\PP_0$ and $\PP_1$ are the observational distributions of Bayesian networks $m_0 = (\PP_0(X_v\given X_{\pa(v)}))_{v\in V}$ and $m_1 = (\PP_1(X_v\given X_{\pa(v)}))_{v\in V}$, this essentially amounts to considering the observational distribution $\PP_\lambda$ of the mixture $(1-\lambda) m_0 + \lambda m_1$ of the Bayesian networks in the sense of Figure \ref{fig:bn_interpolation_3} --- this will explicitly be deployed in Sections \ref{sec:bn_params} and \ref{sec:equicont_params}.
It is therefore immediate that $\PP_\lambda \in M_G$ for all $\lambda \in [0,1]$. Writing $d := |V|$ and letting $(v_1, ..., v_d)$ be a topological ordering of $G$, let $\PP_\alpha := \bigotimes_{i=1}^{d} \PP_{\alpha_i}(X_{v_i}\given X_{\pa(v_i)})$ for $\alpha\in\{0,1\}^d$ denote the observational distribution of the Bayesian network whose Markov kernels are chosen from $\PP_0$ or $\PP_1$ according to $\alpha$. By multilinearity of the product of Markov kernels we have the expansion
\begin{align}
	\begin{split}\label{eqn:expansion}
		\PP_\lambda &= \bigotimes_{i=1}^{d}\left((1-\lambda)\PP_0(X_{v_i} \given X_{\pa(v_i)}) + \lambda \PP_1(X_{v_i} \given X_{\pa(v_i)})\right) = \sum_{\alpha\in \{0,1\}^d} (1-\lambda)^{d-|\alpha|}\lambda^{|\alpha|}\, \PP_\alpha.
	\end{split}
\end{align}
Since each kernel $\PP_{\alpha_i}(X_{v_i}\given x_{\pa(v_i)})$ is dominated by twice the corresponding kernel $\frac{1}{2}(\PP_0 + \PP_1)(X_{v_i}\given x_{\pa(v_i)})$ of $\PP_{1/2}$, iterated integration along the topological ordering gives $\PP_\alpha \leq 2^d\, \PP_{1/2}$ for every $\alpha\in\{0,1\}^d$. Hence, setting $\QQ := \PP_{1/2}$, every $\PP_\alpha$ has a density $p_\alpha$ with respect to $\QQ$, and $\PP_\lambda$ has the density $p_\lambda := \sum_{\alpha\in\{0,1\}^d}(1-\lambda)^{d-|\alpha|}\lambda^{|\alpha|} p_\alpha$ with respect to $\QQ$, for every $\lambda\in[0,1]$. Our goal is to show that if we have conditional dependence $X_A \nIndep_{\PP_1} X_B \given X_C$ in $\PP_1$, then it is maintained in the interpolation $\PP_\lambda$ as $\lambda$ approaches 0. This is not immediate, as shown in the following example.

\begin{example}
	Consider a Bayesian network with variables $X, Y$ taking values in the interval $[-1,1]$ and graph $X\to Y$. Let $\PP_0(X, Y)$ be a uniform distribution on $(0,1)\times (0,1) \cup (-1,0)\times (-1,0)$ and $\PP_1$ a uniform distribution on $(-1,0)\times (0,1) \cup (0,1)\times (-1,0)$. The interpolation $\PP_\lambda$ has a uniform distribution on $(-1,1)^2$ for $\lambda = 1/2$, and thus an independence $X\Indep Y$. This is graphically depicted in Figure \ref{fig:interpolation_independence}.
\end{example}
\begin{figure}[!htb]
	\centering
	\begin{subfigure}{.3\linewidth}
		\centering
		\begin{tikzpicture}
			\draw[-] (-1, 0) -- (1, 0);
			\draw[-] (0, -1) -- (0, 1);
			\fill[gray!50] (0,0) rectangle (1,1);
			\fill[gray!50] (0,0) rectangle (-1,-1);
			\draw[thin, dashed, gray] (-1,-1) grid (1,1);
			\draw (-1,0.05) -- (-1,-0.05) node[left] {-1};
			\draw (1,0.05) -- (1,-0.05) node[right] {1};
			\draw (0.05,-1) -- (-0.05,-1) node[below] {-1};
			\draw (0.05,1) -- (-0.05,1) node[above] {1};
		\end{tikzpicture}
		\caption{$\PP_0: X\nIndep Y$}
	\end{subfigure}
	\begin{subfigure}{.3\linewidth}
		\centering
		\begin{tikzpicture}
			\draw[-] (-1, 0) -- (1, 0);
			\draw[-] (0, -1) -- (0, 1);
			\fill[gray!30] (0,0) rectangle (1,1);
			\fill[gray!30] (0,0) rectangle (-1,-1);
			\fill[gray!30] (-1,0) rectangle (0,1);
			\fill[gray!30] (0,-1) rectangle (1,0);
			\draw[thin, dashed, gray] (-1,-1) grid (1,1);
			\draw (-1,0.05) -- (-1,-0.05) node[left] {-1};
			\draw (1,0.05) -- (1,-0.05) node[right] {1};
			\draw (0.05,-1) -- (-0.05,-1) node[below] {-1};
			\draw (0.05,1) -- (-0.05,1) node[above] {1};
		\end{tikzpicture}
		\caption{$\PP_{\frac{1}{2}}: X\Indep Y$}
	\end{subfigure}
	\begin{subfigure}{.3\linewidth}
		\centering
		\begin{tikzpicture}
			\draw[-] (-1, 0) -- (1, 0);
			\draw[-] (0, -1) -- (0, 1);
			\fill[gray!50] (-1,0) rectangle (0,1);
			\fill[gray!50] (0,-1) rectangle (1,0);
			\draw[thin, dashed, gray] (-1,-1) grid (1,1);
			\draw (-1,0.05) -- (-1,-0.05) node[left] {-1};
			\draw (1,0.05) -- (1,-0.05) node[right] {1};
			\draw (0.05,-1) -- (-0.05,-1) node[below] {-1};
			\draw (0.05,1) -- (-0.05,1) node[above] {1};
		\end{tikzpicture}
		\caption{$\PP_1: X\nIndep Y$}
	\end{subfigure}
	\caption{Mixtures of dependent variables can become independent.}
	\label{fig:interpolation_independence}
\end{figure}

Nevertheless, given $\PP_0$ and $\PP_1$, the dependence is maintained on an interval $(0,\lambda^*) \subset (0,1)$, as shown by the following result.

\begin{lemma}\label{thm:interpolation_dependence}
	Given distributions $\PP_0, \PP_1\in M_G$ with dependence $X_A \nIndep_{\PP_1} X_B \given X_C$, there exists for the interpolation $\PP_\lambda$ from Definition \ref{def:markov_interpolation} a $\lambda^*\in (0,1)$ such that $X_A \nIndep_{\PP_\lambda} X_B \given X_C$ for all $\lambda\in(0,\lambda^*)$.
\end{lemma}
\begin{proof}
	Let $\QQ := \PP_{1/2}$ and let $p_\lambda = \sum_{\alpha}(1-\lambda)^{d-|\alpha|}\lambda^{|\alpha|}p_\alpha$ be the density of $\PP_\lambda$ with respect to $\QQ$ from (\ref{eqn:expansion}). For $S\subseteq V$, we define marginal densities coefficient-wise as $p_\lambda(x_S) := \sum_{\alpha}(1-\lambda)^{d-|\alpha|}\lambda^{|\alpha|}\,\EE_\QQ[p_\alpha(X_V)\given X_S = x_S]$, for fixed versions of the conditional expectations (chosen with values in $[0, 2^d]$, which is possible since $\PP_\alpha \leq 2^d\,\QQ$), and conditional densities as the ratios $p_\lambda(x_S \given x_C) := p_\lambda(x_S, x_C)/p_\lambda(x_C)$ when $p_\lambda(x_C) > 0$, and 0 otherwise. We also fix a disintegration $\QQ(X_A, X_B \given X_C)$, with induced marginals $\QQ(X_A \given X_C)$ and $\QQ(X_B \given X_C)$. Note that conditional dependence $X_A \nIndep_{\PP_1} X_B \given X_C$ is equivalent to the existence of measurable $E_A$ and $E_B$ such that $E_C' := \{x_C : \PP_1(X_A \in E_A, X_B\in E_B \given x_C) \neq \PP_1(X_A \in E_A \given x_C)\PP_1(X_B\in E_B \given x_C)\}$ has $\PP_1(E_C')>0$, hence there is a non-empty measurable subset $E_C \subseteq E_C'$ such that $p_1(x_C) > 0$ for all $x_C\in E_C$, and $\PP_1(E_C) > 0$.
	For $\PP_1$-a.e.\ $x_C$ --- and hence, after removing a $\PP_1$-null set from $E_C$ (which keeps $\PP_1(E_C) > 0$), for every $x_C\in E_C$ ---
	\begin{equation*}
		\PP_1(X_A\in E_A, X_B\in E_B \given X_C = x_C) \neq \PP_1(X_A\in E_A \given X_C = x_C)\PP_1(X_B\in E_B \given X_C = x_C)
	\end{equation*}
	is equivalent to
	\begin{equation*}
		\int_{E_A\times E_B} p_1(x_A, x_B \given x_C) \diff\QQ(x_A, x_B \given x_C) \neq \int_{E_A} p_1(x_A \given x_C) \diff\QQ(x_A \given x_C) \int_{E_B} p_1(x_B \given x_C) \diff\QQ(x_B \given x_C).
	\end{equation*}
	Define
	\begin{multline*}
		q(\lambda, x_C) := \int_{E_A\times E_B} p_\lambda(x_A, x_B, x_C)p_\lambda(x_C) \diff\QQ(x_A, x_B \given x_C) \\
		- \int_{E_A} p_\lambda(x_A, x_C) \diff\QQ(x_A \given x_C) \int_{E_B} p_\lambda(x_B, x_C) \diff\QQ(x_B \given x_C),
	\end{multline*}
	for which we have $q(1, x_C) \neq 0$ for all $x_C \in E_C$.
	Then $q(\lambda, x_C)$ is a non-zero polynomial in $\lambda$ for every $x_C \in E_C$, and so $q(\lambda, x_C) \neq 0$ for all $\lambda \in (0,\lambda^*(x_C))$ with $\lambda^*(x_C) := \min \{\{1\}\cup R(x_C)\}$ where $R(x_C) := \{\lambda \in (0,1] : q(\lambda, x_C) = 0\}$. Since $q(\lambda, x_C)$ is a \emph{Carathéodory function} (continuous in $\lambda$ and measurable in $x_C$), we have by \cite{aliprantis2006infinite}, Corollary 18.8, that the correspondence $\bar{R}(x_C) := \{\lambda \in [0,1] : q(\lambda, x_C) = 0\}$ is \emph{weakly measurable} (i.e.\ $\{x_C : \bar{R}(x_C) \cap A \neq \emptyset\} \in \Bcal(\Xcal_C)$ for all open $A$). For each open $A$ we have $R(x_C) \cap A = \bar{R}(x_C)\cap (A\setminus \{0\})$, so the correspondence $R(x_C)$ is weakly measurable as well. It then follows from \cite{aliprantis2006infinite}, Theorem 18.19 that the selection $\lambda^*(x_C)$ is measurable.
	Our goal is to show that there is a $\lambda^* \in (0,1)$ (independent of $x_C$) and a set $E_C^*\in \Bcal(\Xcal_C)$ with $\PP_\lambda(E_C^*) > 0$ and $q(\lambda, x_C)\neq 0$ for all $\lambda \in(0,\lambda^*)$ and all $x_C \in E_C^*$, which would imply that $X_A\nIndep_{\PP_\lambda} X_B \given X_C$ for all $\lambda \in (0,\lambda^*)$.\footnote{\label{footnote}It would be more straightforward to show that $\PP_\lambda(X_A\in E_A, X_B\in E_B\given X_C\in E_C) \neq \PP_\lambda(X_A\in E_A\given X_C\in E_C)\PP_\lambda(X_B\in E_B\given X_C\in E_C)$ for some $E_A, E_B, E_C$ with $\PP_\lambda(E_C)>0$, but this does not imply the conditional dependence $X_A\nIndep_{\PP_\lambda} X_B \given X_C$, hence the conditioning on the individual $x_C \in E_C$. See also \cite{neykov2021minimax}, p.3.}
	Define $E_C^n := \{x_C \in E_C : \lambda^*(x_C) > 1/n\}$, then $E_C^1\subseteq E_C^2 \subseteq ... \subseteq E_C$ with $\lim_n\PP_1(E_C^n) = \PP_1(E_C) > 0$, so there exists an $N$ such that $\PP_1(E_C^n) > 0$ for all $n\geq N$. Setting $\lambda^* := 1/N$ and $E_C^* := E_C^N$ we get $q(\lambda, x_C) \neq 0$ for all $\lambda \in (0,\lambda^*)$ for all $x_C \in E_C^*$. Note from (\ref{eqn:expansion}) that $p_\lambda \geq \lambda^d p_1$ pointwise, since all terms are non-negative and the coefficient of $p_1 = p_{(1,...,1)}$ is $\lambda^d$; the same bound holds for the marginal densities. Hence $\PP_1$ is absolutely continuous with respect to $\PP_\lambda$ for all $\lambda \in (0,1)$, so we also have $\PP_\lambda(E_C^*) > 0$, and moreover $p_\lambda(x_C) \geq \lambda^d p_1(x_C) > 0$ on $E_C^*$. This implies $X_A \nIndep_{\PP_\lambda} X_B \given X_C$ for all $\lambda \in (0,\lambda^*)$, which is the desired result.
\end{proof}

\begin{lemma}\label{thm:interpolation_convergence}
	Given two distributions $\PP_0, \PP_1\in M_G$, we have for the interpolation $\PP_\lambda$ from Definition \ref{def:markov_interpolation} that $\PP_\lambda\tvto \PP_0$ as $\lambda\to 0$.
\end{lemma}
\begin{proof}
	Let $\QQ := \PP_{1/2}$ and let $p_\alpha$ and $p_\lambda$ be as in (\ref{eqn:expansion}), so that
	\begin{equation*}
		p_\lambda(x_V) = (1-\lambda)^{d}p_0(x_V) + \sum_{\substack{\alpha\in \{0,1\}^d\\|\alpha|>0}} (1-\lambda)^{d-|\alpha|}\lambda^{|\alpha|}p_{\alpha}(x_V)
	\end{equation*}
	with $p_0 = p_{(0,...,0)}$ the density of $\PP_0$, and we have pointwise convergence $p_\lambda(x_V) \to p_0(x_V)$ as $\lambda \to 0$. By \cite{scheffe1947useful} we conclude that $\PP_\lambda \tvto \PP_0$.
\end{proof}

\subsection{Typicality of faithfulness in the space of Bayesian networks}\label{sec:bn_params}
In this section we extend Theorem \ref{thm:bn_dist_nowhere_dense} from the space of observational distributions to the space of Bayesian networks, defined as follows:
\begin{definition}
	Given a DAG $G$ with finite index set $V$ and standard Borel space $\Xcal_v$ for every $v\in V$, \emph{the space of Bayesian networks with graph $G$} is defined as
	\begin{equation*}
		\BN_G := \prod_{v\in V}\left\{\PP(X_v \given X_{\pa(v)}) : \Xcal_{\pa(v)} \to \Pcal(\Xcal_v) \text{ measurable}\right\}.
	\end{equation*}
\end{definition}
The faithfulness of a Bayesian network is a property of its observational distribution $\PP \in M_G$. To formalise the relation between the Bayesian network and the observational distribution we introduce the \emph{distribution} map, defined as
\begin{equation*}
	D: \BN_G \to M_G, \quad (\PP(X_v\given X_{\pa(v)}))_{v\in V} \mapsto \bigotimes_{v\in V} \PP(X_v \given X_{\pa(v)}).
\end{equation*}
We are interested in whether the faithful Bayesian networks $D^{-1}(F_G)$ are typical in $\BN_G$. To get a well-defined notion of typicality we require a topology on $\BN_G$, which we introduce via the following metric.
\begin{definition}
	For $m, m' \in \BN_G$, consider the following metric:
	\begin{equation*}
		d_{TV}^\circ(m, m') := \sum_{v\in V}\sup_{x_{\pa(v)}}d_{TV}(\PP_m(X_v\given x_{\pa(v)}), \PP_{m'}(X_v\given x_{\pa(v)})).
	\end{equation*}
	One readily verifies that this is indeed a metric.
\end{definition}
This metric measures the worst-case total variation distance between corresponding Markov kernels, summed over all vertices. Convergence $d_{TV}^\circ(m_n, m)\to 0$ means that all conditional distributions $\PP_{m_n}(X_v\given x_{\pa(v)})$ converge uniformly in $x_{\pa(v)}$ to $\PP_m(X_v\given x_{\pa(v)})$ in total variation.
Note that this is a strong notion of closeness: it requires agreement of the Markov kernels for \emph{all} parent values, including those that might have zero probability under the observational distribution. Consequently, our metric distinguishes Bayesian networks that have the same observational distribution but differ on a measure-zero set of parent values. This metric is natural for Bayesian networks as causal models, where by intervening the Markov kernels represent mechanisms which could be evaluated on \emph{all} parent values --- not just those occurring under a specific distribution of the parent variables.

\begin{lemma}\label{thm:dist_map_continuous}
	The distribution map $D: (\BN_G, d_{TV}^\circ) \to (M_G, d_{TV})$ is continuous.
\end{lemma}
\begin{proof}
	Let $m_0, m_1\in \BN_G$ and write $\PP_0 := D(m_0)$ and $\PP_1 := D(m_1)$. Let $v_1, ..., v_d$ be a reverse topological ordering of $G$ (which has $\pa(v_k)\subseteq v_{k+1}, ..., v_d$), and let $\QQ_k := \bigotimes_{i=1}^k \PP_1(X_{v_i}\given X_{\pa(v_i)}) \otimes \bigotimes_{i=k+1}^d \PP_0(X_{v_i}\given X_{\pa(v_i)})$, then $\QQ_0 = \PP_0$ and $\QQ_d = \PP_1$, so we have $d_{TV}(\PP_0, \PP_1) \leq \sum_{k=1}^d d_{TV}(\QQ_{k-1}, \QQ_{k})$ by the triangle inequality.
	Given a measurable function $f: \Xcal_V\to [-1,1]$, define the functions
	\begin{align*}
		g_f^k(x_{v_{k}}, ..., x_{v_{d}}) & := \int f(x_{v_1}, ..., x_{v_d})\diff \bigotimes_{i=1}^{k-1}\PP_1(X_{v_i}\given X_{\pa(v_i)})                          \\
		h_f^k(x_{v_k}, x_{\pa(v_{k})})   & := \int g_f^k(x_{v_k}, x_{\pa(v_k)}, ...)\diff \PP_0(X_{\{v_{k+1}, ..., v_d\}\setminus \pa(v_k)} \given x_{\pa(v_k)}),
	\end{align*}
	which are bounded and measurable. Note that $\QQ_{k-1}$ and $\QQ_{k}$ only differ at the Markov kernel for $v_k$, having $\PP_0(X_{v_k} \given X_{\pa(v_k)})$ and $\PP_1(X_{v_k} \given X_{\pa(v_k)})$ at that position respectively. By taking the supremum over all measurable $f:\Xcal_V \to [-1,1]$ we then bound $2 d_{TV}(\QQ_{k-1}, \QQ_{k}) $ as
	\begin{align*}
		 & \sup_f \left|\int f\left(\diff\QQ_{k-1} - \diff\QQ_{k}\right)\right|                                                                                                                                                                                \\
		 & = \sup_f \left|\int\left(\int g_f^k \left(\diff\PP_0(X_{v_k} \given X_{\pa(v_k)}) - \diff\PP_1(X_{v_k} \given X_{\pa(v_k)})\right)\right) \diff\PP_0(X_{v_{k+1}}, ..., X_{v_d})\right|                                                              \\
		 & \leq \sup_{x_{\pa(v_k)}} \sup_f \left|\int\left(\int g_f^k \left(\diff\PP_0(X_{v_k} \given x_{\pa(v_k)}) - \diff\PP_1(X_{v_k} \given x_{\pa(v_k)})\right)\right) \diff\PP_0(X_{\{v_{k+1}, ..., v_d\}\setminus \pa(v_k)} \given x_{\pa(v_k)})\right| \\
		 & = \sup_{x_{\pa(v_k)}} \sup_f \left|\int h_f^k(x_{v_k}, x_{\pa(v_k)})\left(\diff\PP_0(X_{v_k} \given x_{\pa(v_k)}) - \diff\PP_1(X_{v_k} \given x_{\pa(v_k)})\right)\right|                                                                           \\
		 & \leq 2 \sup_{x_{\pa(v_k)}}d_{TV}(\PP_0(X_{v_k} \given x_{\pa(v_k)}),  \PP_1(X_{v_k} \given x_{\pa(v_k)})),
	\end{align*}
	using the disintegration of $\PP_0(X_{v_{k+1}}, ..., X_{v_d})$ into $\PP_0(X_{\{v_{k+1}, ..., v_d\}\setminus \pa(v_k)} \given X_{\pa(v_k)})$ and $\PP_0(X_{\pa(v_k)})$ for the first inequality, and Fubini's Theorem in the second equality. This gives $d_{TV}(\PP_0(X_V), \PP_1(X_V)) \leq d_{TV}^\circ(m_0, m_1)$.
\end{proof}

Contrasting with the typicality of faithful distributions of Bayesian networks, we now show that Bayesian networks themselves are typically faithful.
\begin{theorem}\label{thm:bn_params_nowhere_dense}
	The set of faithful Bayesian networks $D^{-1}(F_G)$ is open and dense in $(\BN_G, d_{TV}^\circ)$.
\end{theorem}
\begin{proof}
	From Lemma \ref{thm:dist_map_continuous} and Theorem \ref{thm:bn_dist_nowhere_dense} it follows that the set of faithful Bayesian networks $D^{-1}(F_G)$ is open and non-empty. Let $m_0 \in \BN_G$, $m_1 \in D^{-1}(F_G)$ and $m_\lambda := (1-\lambda)m_0 + \lambda m_1$, then
	\begin{align*}
		d_{TV}^\circ(m_\lambda, m_0) & = \sum_{v\in V} \sup_{x_{\pa(v)}} d_{TV}(\PP_{\lambda}(X_v\given x_{\pa(v)}), \PP_{0}(X_v\given x_{\pa(v)})) \\
		                             & = \lambda \sum_{v\in V} \sup_{x_{\pa(v)}} d_{TV}(\PP_1(X_v\given x_{\pa(v)}), \PP_{0}(X_v\given x_{\pa(v)})) \\
		                             & = \lambda d_{TV}^\circ(m_1, m_0)\to 0
	\end{align*}
	as $\lambda \to 0$. By repeated application of Lemma \ref{thm:interpolation_dependence} for every $d$-connection in $G$, there is a $\lambda^*\in(0,1)$ such that $D(m_\lambda) \in F_G$ for all $\lambda \in (0,\lambda^*)$, so $D^{-1}(F_G)$ is dense in $\BN_G$.
\end{proof}

\section{Conditional exponential family parametrisations}\label{sec:exp}
The preceding section raises the question whether the topological typicality of faithful Bayesian networks also holds for specific parametrisations of Bayesian networks. In this section we answer this question in the affirmative, by extending the results of \cite{spirtes1993causation} and \cite{meek1995strong} to sufficiently regular conditional exponential family parametrisations of Bayesian networks.

Formally, a \emph{parametrisation} of a Bayesian network with graph $G$ is a set $\Theta\subseteq \RR^d$ for some $d\in\NN$ and a map
\begin{equation*}
	\varphi: \Theta \to \BN_G, \quad \theta \mapsto (\PP_\theta(X_v\given X_{\pa(v)}))_{v\in V}.
\end{equation*}
The corresponding map from the parameter values to the observational distribution is defined as
\begin{equation*}
	T:\Theta \to M_G, \quad T:= D\circ \varphi.
\end{equation*}
We consider the question whether the set of \emph{faithful parameters} $T^{-1}(F_G)$ is typical in $\Theta$ with respect to the Euclidean topology, in particular for the following class of conditional exponential family Bayesian networks, inspired by \cite{feigin1981conditional}.
\begin{definition}
	Given a DAG $G$ with vertices $V$ and sample space $\Xcal_v \subseteq \RR$ for each $v\in V$, a \emph{conditional exponential family Bayesian network} has for each $v\in V$ a conditional density
	\begin{equation*}
		p_{\theta_v}(x_v\given x_{\pa(v)}) = b_v(x_v, x_{\pa(v)})e^{\eta_v(\theta_v)^\top t_v(x_v, x_{\pa(v)}) - A_v(\eta_v(\theta_v), x_{\pa(v)})}
	\end{equation*}
	with respect to a given locally finite dominating measure $\mu_v$ on $\Xcal_v$, with parameter space $\Theta_v\subseteq \RR^{d_v}$ for some $d_v\in\NN$, a function $b_v : \Xcal_v\times\Xcal_{\pa(v)}\to [0,\infty)$, \emph{natural parameter} $\eta_v : \Theta_v \to \RR^{k_v}$ for some $k_v\in\NN$ and \emph{sufficient statistic} $t_v : \Xcal_v\times \Xcal_{\pa(v)} \to \RR^{k_v}$ such that $A_v(\eta_v(\theta_v), x_{\pa(v)}) < \infty$ for all $\theta_v\in \Theta_v$ and $x_{\pa(v)}\in\Xcal_{\pa(v)}$, where $A_v(\eta_v, x_{\pa(v)}) := \log\int b_v(x_v, x_{\pa(v)})e^{\eta_v^\top t_v(x_v, x_{\pa(v)})}\diff \mu_v(x_v)$.
\end{definition}
This gives rise to a joint density $p_\theta(x_V) = \prod_{v\in V}p_{\theta_v}(x_v\given x_{\pa(v)})$ of the distribution $\PP_{\theta}(X_V) := T(\theta) \in M_G$ and the joint parameter space $\Theta := \prod_{v\in V}\Theta_v$. This model class allows the modelling of mixed data types, see e.g.\ \cite{yang2014mixed}.

\begin{example}\label{ex:exp_gauss}
	For linear Gaussian Bayesian networks, \cite{spirtes1993causation} parametrise for each $v\in V$ the conditional distribution $\PP_\theta(X_v \given X_{\pa(v)} = x_{\pa(v)}) = \Ncal(\beta_v^\top x_{\pa(v)}, \sigma_v^2)$ by a linear coefficient $\beta_v$ and a strictly positive variance $\sigma_v^2$.
	This gives the parameter space
	\begin{equation*}
		\Theta_\Ncal := \prod_{v\in V} \left\{(\beta_v, \sigma_v^2) \in \RR^{|\pa(v)|} \times \RR_{>0}\right\},
	\end{equation*}
	so when writing $\theta_v = (\beta_v, \sigma_v^2)$ it has sufficient statistic $t_v(x_v, x_{\pa(v)})^\top = (x_vx^\top_{\pa(v)}, x_v^2)$, natural parameter $\eta_v(\theta_v)^\top = (\beta_v^\top/\sigma_v^2, -1/(2\sigma_v^2))$, $b_v = 1$ and as dominating measure $\mu_v$ the Lebesgue measure.
\end{example}

\begin{example}\label{ex:exp_discrete}
	For discrete distributions with finite state space, \cite{meek1995strong} considers for each conditional distribution $\PP_\theta(X_v \given X_{\pa(v)} = x_{\pa(v)})$ a probability vector $(\theta_{v, x_{\pa(v)}, x_v})_{x_v\in\Xcal_v}$. Fixing a reference state $x_v^0\in\Xcal_v$, we parametrise it by its coordinates $\theta_{v, x_{\pa(v)}} := (\theta_{v, x_{\pa(v)}, x_v})_{x_v \neq x_v^0}$ in the simplex $\Delta_v := \{\vartheta \in [0,1]^{|\Xcal_v|-1} : \sum_j \vartheta_j \leq 1\}$, setting $\theta_{v, x_{\pa(v)}, x_v^0} := 1 - \sum_{x_v\neq x_v^0}\theta_{v, x_{\pa(v)}, x_v}$. This gives the parameter space
	\begin{equation*}
		\Theta_\Dcal := \prod_{v\in V} \prod_{x_{\pa(v)} \in \Xcal_{\pa(v)}} \left\{\theta_{v, x_{\pa(v)}} \in \Delta_v\right\}.
	\end{equation*}
	The sufficient statistic is the vector $t_v(x_v, x_{\pa(v)})$ of length $|\Xcal_{\pa(v)}| (|\Xcal_{v}|-1)$ with entry 1 at the $(x_{\pa(v)}, x_v)$ position for $x_v \neq x_v^0$ and zeros elsewhere, the natural parameter $\eta_v(\theta_v)$ is given by the vector with entry $\log(\theta_{v, x_{\pa(v)}, x_v}/\theta_{v, x_{\pa(v)}, x_v^0})$ for every pair $(x_{\pa(v)}, x_v)$ with $x_v \neq x_v^0$, $b_v = 1$, and as dominating measure $\mu_v$ the counting measure.
\end{example}

\subsection{Typicality of faithfulness in the parameter space}\label{sec:exp_params}
To obtain that faithful parameters are typical we require that the conditional independence constraints of faithfulness violations are highly restrictive. In Theorem \ref{thm:bn_dist_nowhere_dense} (in particular, Lemma \ref{thm:interpolation_dependence}) we leveraged that conditional independence is a polynomial constraint in the interpolation parameter, but this approach is not feasible for conditional exponential families since these are not closed under taking mixtures. Instead, we ensure that each marginal density $p_\theta(x_A)$ is an analytic function in the parameter $\theta$, and hence that conditional independence is the zero set of an analytic function, which is atypical in the parameter space.

\begin{theorem}\label{thm:margdensity_params_nowhere_dense}
	Given DAG $G$ with vertices $V$, let $\mu_v$ be a $\sigma$-finite measure for each $v\in V$ and let $\varphi : \Theta \to \BN_G$ be a Bayesian network parametrisation with $\Theta$ open and connected, such that for each $A\subseteq V$ the marginal distribution $\PP_\theta(X_A)$ has a density $p_\theta(x_A)$ with respect to $\mu_A := \bigotimes_{v\in A}\mu_v$, and such that there is a $\mu_A$-null set $N_A$ with the property that for every $x_A\notin N_A$ the map $\theta\mapsto p_\theta(x_A)$ is analytic, and for every $\theta\in\Theta$ the function $x_A'\mapsto p_\theta(x_A')$ is continuous at $x_A$. If there is at least one faithful parameter in $\Theta$, then the set of faithful parameters is open and dense in $\Theta$, and the unfaithful parameters have Lebesgue measure zero.
\end{theorem}
\begin{proof}
	Let $A, B, C\subseteq V$ such that $A\nperp_G^d B\given C$, and let $\theta_1\in \Theta$ be a parameter such that $X_A\nIndep_{\PP_{\theta_1}} X_B \given X_C$, which exists by assumption. The conditional dependence implies that the set of points where $p_{1}(x_A, x_B, x_C)p_{1}(x_C) - p_{1}(x_A, x_C)p_{1}(x_B, x_C)\neq 0$ has positive $\mu_{A\cup B\cup C}$-measure. The sets $N_{A\cup B\cup C}$, $N_{A\cup C}\times\Xcal_{B\setminus(A\cup C)}$, $N_{B\cup C}\times\Xcal_{A\setminus(B\cup C)}$ and $N_C\times\Xcal_{(A\cup B)\setminus C}$ (up to reordering of coordinates) are $\mu_{A\cup B\cup C}$-null by Tonelli's theorem, so we can pick a point $(x_A, x_B, x_C)$ in the support of $\mu_{A\cup B\cup C}$, outside these four null sets, such that $p_{1}(x_A, x_B, x_C)p_{1}(x_C) - p_{1}(x_A, x_C)p_{1}(x_B, x_C)\neq 0$, and define
	\begin{equation*}
		q(\theta) :=  p_{\theta}(x_A, x_B, x_C)p_{\theta}(x_C) - p_{\theta}(x_A, x_C)p_{\theta}(x_B, x_C).
	\end{equation*}
	By the choice of $(x_A, x_B, x_C)$ we have $q(\theta_1) \neq 0$, and $\theta\mapsto q(\theta)$ is analytic. For every $\theta$ such that $X_A\Indep_{\PP_\theta} X_B\given X_C$, the factorisation $p_\theta(x_A', x_B', x_C')p_\theta(x_C') = p_\theta(x_A', x_C')p_\theta(x_B', x_C')$ holds for $\mu_{A\cup B\cup C}$-a.e.\ point $(x_A', x_B', x_C')$; since both sides are continuous at $(x_A, x_B, x_C)$ and every neighbourhood of this point has positive $\mu_{A\cup B\cup C}$-measure (as it lies in the support of $\mu_{A\cup B\cup C}$), the factorisation holds at $(x_A, x_B, x_C)$ itself, i.e.\ $q(\theta)=0$. Hence $\{\theta \in \Theta : X_A\Indep_{\PP_\theta} X_B\given X_C\}\subseteq \{\theta \in \Theta : q(\theta) = 0\}$. It follows from the identity theorem\footnote{Every real analytic function on a connected open domain in $\RR^n$ can be extended to a holomorphic function on a connected open domain in $\CC^n$, for which the identity theorem of holomorphic functions applies \citep{gunning1965analytic}. The identity theorem for functions of real variables is obtained by restricting again to $\RR^n$.} that the zero set of a nonconstant real analytic function on an open and connected domain is nowhere dense (if it were dense on an open set, then by continuity the function would be 0 on that open set, and hence zero on the whole domain), and has Lebesgue measure zero \citep{mityagin2020zero}. Hence, $T^{-1}(U_G) = \bigcup_{A\nperp^d_G B\given C}\{\theta \in \Theta: X_A \Indep_{\PP_\theta} X_B\given X_C\}$ is nowhere dense and has Lebesgue measure zero. Finally, since the unfaithful parameters are nowhere dense, the set of faithful parameters is dense. Since $\theta \mapsto p_\theta(x_V)$ is continuous for $\mu_V$-a.e.\ $x_V$, it follows from Scheffé's Theorem that $\theta \mapsto \PP_\theta(X_V)$ is continuous with respect to the total variation metric, so by Theorem \ref{thm:ci_is_closed} the set of faithful parameters is open.
\end{proof}

\begin{remark}
	The proof of Theorem \ref{thm:margdensity_params_nowhere_dense} only requires that for every $d$-connection $A\nperp^d_G B\given C$ in the graph there is a parameter $\theta\in \Theta$ such that $X_A\nIndep_{\PP_\theta} X_B\given X_C$. Given the required analyticity, it follows from the preceding theorem that this is equivalent to the existence of a parameter that is faithful. For a specific parametrisation, the former condition might be easier to prove than the latter --- this strategy has also been employed in the original proofs of Theorems \ref{thm:spirtes} and \ref{thm:meek}.
\end{remark}

The following result shows that for conditional exponential family parametrisations whose joint distribution lies in a regular exponential family, the analyticity property of the marginal $p_\theta(x_A)$ is ensured if the parametrisation has analytic natural parameters $\eta_v$. An exponential family is \emph{regular} if the density $p_\theta(x_V)$ (with respect to $\mu_V$) is proportional to $\tilde{b}(x_V)e^{\tilde{\eta}(\theta)^\top \tilde{t}(x_V)}$ for some non-negative function $\tilde{b}$ and vector-valued functions $\tilde{\eta} : \Theta \to \RR^k$ and $\tilde{t} : \Xcal_V \to \RR^k$ such that the parametrisation is \emph{minimal} (i.e., the components of $\tilde{\eta}$ are affinely independent and the components of $\tilde{t}$ are $\tilde b\mu$-a.s.\ affinely independent), \emph{full} (the natural parameters satisfy $\tilde{\eta}(\Theta) = \{\gamma : \int \tilde{b}(x_V)e^{\gamma^\top \tilde{t}(x_V)}\diff\mu(x_V) < \infty\}$) and $\tilde{\eta}(\Theta)$ is open --- see \citealp{barndorff-nielsen2014information} (page 116). We write $\tilde{A}(\gamma) := \log\int \tilde{b}(x_V)e^{\gamma^\top \tilde{t}(x_V)}\diff\mu(x_V)$ for the log-partition function. The joint densities $\{p_\theta : \theta\in\Theta\}$ \emph{lie in a regular exponential family} if they form a subset of the densities of a regular exponential family $\{p_{\tilde{\theta}} : \tilde{\theta}\in\tilde{\Theta}\}$.

\begin{theorem}\label{thm:exp_params_nowhere_dense}
	Given DAG $G$ with vertices $V$, consider a conditional exponential family parametrisation of the Bayesian network such that for each $v\in V$ the set $\Theta_v$ is open and connected and the natural parameter $\theta\mapsto \eta_v(\theta)$ is analytic with open image $\eta_v(\Theta_v)$, and the joint densities $p_\theta(x_V)$ lie in a regular exponential family such that for every $A\subseteq V$ there is a $\mu_A$-null set $N_A$ such that for every $x_A \notin N_A$ and every $\theta\in\Theta$ the marginal density $x_A' \mapsto p_\theta(x_A')$ is continuous at $x_A$. If there is at least one faithful parameter in $\Theta$, then the set of faithful parameters is open and dense, and the unfaithful parameters have Lebesgue measure zero.
\end{theorem}
\begin{proof}
	The map $\gamma \mapsto e^{A_v(\gamma, x_{\pa(v)})} = \int_{\Xcal_v} b_v(x_v, x_{\pa(v)})e^{\gamma^\top t_v(x_v, x_{\pa(v)})}\diff\mu_v(x_v)$ can alternatively be written as the complex Fourier-Laplace transform of the push-forward measure $d\nu := d (t_v)_*(b_v\mu_v)$ given by $\gamma \mapsto\int e^{\gamma^\top y}\diff \nu (y)$, defined on $\Theta_v' := \{\gamma + i \zeta: \gamma \in \eta_v(\Theta_v), \zeta \in \RR^{k_v} \}$. Since the range of $\eta_v$ is open, it follows from \cite{barndorff-nielsen2014information}, Theorem 7.2 that this map is holomorphic on $\Theta_v'$.\footnote{Formally \cite{barndorff-nielsen2014information}, Theorem 7.2 assumes $\nu$ to be a probability measure, but the result nevertheless holds when $\nu$ is $\sigma$-finite.} As a composition of analytic functions, $p_{\theta_v}(x_v\given x_{\pa(v)})$ is analytic in $\theta$. As a product of analytic functions, the joint density $p_\theta(x_V) = \tilde{b}(x_V)e^{\tilde{\eta}(\theta)^\top \tilde{t}(x_V) - \tilde{A}(\tilde{\eta}(\theta))}$ is analytic in $\theta$ as well. Fix any $x_V^0$ with $\tilde{b}(x_V^0) > 0$. The vectors $\{\tilde{t}(x_V) - \tilde{t}(x_V^0) : \tilde{b}(x_V) > 0\}$ span $\RR^k$, since otherwise some nonzero $a$ would satisfy $a^\top\tilde{t}(x_V) = a^\top\tilde{t}(x_V^0)$ for all $x_V$ with $\tilde{b}(x_V) > 0$, contradicting minimality. Hence there are $x_V^1, ..., x_V^k$ with $\tilde{b}(x_V^i) > 0$ such that the vectors $u_i := \tilde{t}(x_V^i) - \tilde{t}(x_V^0)$ are linearly independent, and $p_\theta(x_V^i) > 0$ for all $\theta\in\Theta$ and $i = 0, ..., k$. We can write $U^\top \tilde{\eta}(\theta) = g(\theta)$ with $U := \left(u_1, ..., u_k\right)$ invertible and $g_i(\theta) := \log(p_\theta(x_V^i)) - \log(p_\theta(x_V^0)) + C_i$ analytic for some constant $C_i$, hence $\tilde{\eta}(\theta)$ is analytic as well, as a sum of products of analytic functions.
	Next, fix $I\subseteq V$ and write $F_I(\gamma, x_{V\setminus I}) := \int_{\Xcal_I}\tilde{b}(x_V)e^{\gamma^\top \tilde{t}(x_V)}\diff\mu(x_I)$, which for fixed $x_{V\setminus I}$ is the Laplace transform of the push-forward measure $(\tilde{t}(\cdot, x_{V\setminus I}))_*(\tilde{b}(\cdot, x_{V\setminus I})\,\mu_I)$. For fixed $\gamma\in\Gamma$, Tonelli's theorem gives $\int F_I(\gamma, x_{V\setminus I})\diff\mu(x_{V\setminus I}) = e^{\tilde{A}(\gamma)} < \infty$, so $F_I(\gamma, x_{V\setminus I}) < \infty$ for $\mu_{V\setminus I}$-a.e.\ $x_{V\setminus I}$. Let $D\subseteq\Gamma$ be countable and dense, and let $N_{V\setminus I}'$ be the $\mu_{V\setminus I}$-null set of those $x_{V\setminus I}$ with $F_I(\gamma, x_{V\setminus I}) = \infty$ for some $\gamma\in D$. Since $\Gamma$ is open, every $\gamma\in\Gamma$ lies in the convex hull of finitely many points of $D$ (see \citealp{barndorff-nielsen2014information}, Theorems 5.1 and 5.3), and since the effective domain of a Laplace transform is convex (\citealp{barndorff-nielsen2014information}, Theorem 7.1), we have $F_I(\gamma, x_{V\setminus I}) < \infty$ for all $\gamma\in\Gamma$ whenever $x_{V\setminus I}\notin N_{V\setminus I}'$. For such $x_{V\setminus I}$, the same argument as before shows that $\gamma\mapsto F_I(\gamma, x_{V\setminus I})$ is analytic on the open set $\Gamma\supseteq\tilde{\eta}(\Theta)$. Applying this with $I = V\setminus A$ shows that the unnormalised density $\tilde{p}_\theta(x_A) = F_{V\setminus A}(\tilde{\eta}(\theta), x_A)$ is analytic in $\theta$ for every $x_A$ outside the $\theta$-independent $\mu_A$-null set $N_A'$, and applying it with $I = V$ shows that the function $\tilde{A}$ is analytic on $\Gamma$. Hence $p_\theta(x_A) = \tilde{p}_\theta(x_A)e^{-\tilde{A}(\tilde{\eta}(\theta))}$ is analytic in $\theta$ for every $x_A\notin N_A'$.
	The final claim now follows from Theorem \ref{thm:margdensity_params_nowhere_dense}, applied with the null set $N_A \cup N_A'$.
\end{proof}

\begin{remark}\label{rem:exp_faithful_typical_convex}
	When $\Theta$ is the closure of an open convex set, a statement similar to Theorem \ref{thm:exp_params_nowhere_dense} holds. Since the boundary of an open convex set is nowhere dense and has Lebesgue measure zero \citep{lang1986note}, we obtain that if there is a faithful parameter in $\mathrm{int}(\Theta)$, the unfaithful parameters are nowhere dense and have Lebesgue measure zero in $\mathrm{int}(\Theta)$, hence also in $\Theta$.
\end{remark}

For conditional exponential family parametrisations, the joint distribution does not automatically lie in a regular exponential family. For example, when $Y|X=x \sim \textrm{Gamma}(\theta_1 x, 1)$ and $X\sim \textrm{Exp}(\theta_2)$, the joint density is proportional to $\exp(-\theta_2 x + (\theta_1 x -1) \log(y) - y - \log(\Gamma(\theta_1 x)))$, and this exponent cannot be written as the product of a function of the parameters and a function of the data, so it's not in an exponential family. Positive examples are the discrete and Gaussian Bayesian networks, or the network specified by $Y|X=x \sim \textrm{Exp}(\theta_1 x)$ and $X\sim \textrm{Exp}(\theta_2)$.

\vspace{1em}
Given that \cite{spirtes1993causation} and \cite{meek1995strong} have proven for every DAG $G$ the existence of faithful parameters in $\Theta_\Ncal$ and in the interior of $\Theta_\Dcal$, and given that these models induce joint densities which lie in regular exponential families, we obtain Theorems \ref{thm:spirtes} and \ref{thm:meek} and their topological analogues as corollaries of Theorem \ref{thm:exp_params_nowhere_dense} and Remark \ref{rem:exp_faithful_typical_convex}:

\begin{corollary}\label{thm:spirtes_nowhere_dense}
	The set of faithful parameters $\{\theta \in \Theta_\Ncal : T(\theta) \in F_G\}$ is open and dense and the set of unfaithful parameters $\{\theta \in \Theta_\Ncal : T(\theta) \in U_G\}$ has Lebesgue measure zero.
\end{corollary}
\begin{corollary}\label{thm:meek_nowhere_dense}
	The set of faithful parameters $\{\theta \in \Theta_\Dcal : T(\theta) \in F_G\}$ is open and dense\footnote{That this set is open and dense has also been shown by \cite{lin2020learning}. From Remark \ref{rem:exp_faithful_typical_convex} we can merely conclude that the unfaithful parameters are nowhere dense, but since by Scheffé's theorem the map $T$ is continuous, the set of unfaithful parameters is \emph{closed} and nowhere dense, hence the faithful parameters are open and dense.} and the set of unfaithful parameters $\{\theta \in \Theta_\Dcal : T(\theta) \in U_G\}$ has Lebesgue measure zero.
\end{corollary}


\subsection{Typicality of faithfulness in the space of observational distributions}\label{sec:exp_dists}
Similar to Section \ref{sec:bn_dists}, we also investigate the typicality of faithful observational distributions. This is considered with respect to the set of distributions induced by a given conditional exponential family parametrisation $(\Theta, b, \eta, t, \mu)$, denoted by $M_G^{\mathrm{exp}} := T(\Theta)$. The faithful distributions are denoted by $F_G^{\mathrm{exp}} := T(\Theta)\cap F_G$. For this set of observational distributions, the total variation topology coincides with the weak topology, so in contrast with the result in Section \ref{sec:bn_dists}, we now obtain that faithfulness is open and dense in the weak topology.
\begin{theorem}\label{thm:exp_dist_nowhere_dense}
	Consider a conditional exponential family Bayesian network parametrisation satisfying the assumptions of Theorem \ref{thm:exp_params_nowhere_dense}, then the total variation topology and the weak topology coincide on $M_G^{\mathrm{exp}}$. If there is at least one faithful parameter in $\Theta$, then the set of faithful distributions $F_G^{\mathrm{exp}}$ is open and dense in $M_G^{\mathrm{exp}}$, in either topology.
\end{theorem}
\begin{proof}
	Let $\tilde{\eta}(\theta)$ denote the natural parameter of the minimal parametrisation of the joint density $p_\theta(x_V)$, and let $\tilde{t}$ denote the sufficient statistic. Without loss of generality the affine hull of $\tilde{\eta}(\Theta)$ is all of $\RR^{k}$: otherwise, replace the ambient family by its restriction to $H$, defined as the affine hull of $\tilde{\eta}(\Theta)$, which can be reparametrised over $\RR^{\dim H}$ as a regular exponential family in canonical form: writing $H = \gamma_0 + \mathrm{ran}(M)$ with $M$ injective and substituting $\gamma = \gamma_0 + Ms$, the restricted family has representation $\tilde{b}'(x_V)e^{s^\top \tilde{t}'(x_V) - \tilde{A}'(s)}$ with base function $\tilde{b}' := \tilde{b}\,e^{\gamma_0^\top\tilde{t}}$ and statistic $\tilde{t}' := M^\top\tilde{t}$, leaving the densities unchanged, and the reparametrised exponential family remains regular.
	Since the affine hull of $\tilde{\eta}(\Theta)$ is $\RR^k$, there are $\theta^0, \theta^1, ..., \theta^k\in\Theta$ such that the matrix $W := (\tilde{\eta}(\theta^1) - \tilde{\eta}(\theta^0), ..., \tilde{\eta}(\theta^k) - \tilde{\eta}(\theta^0))$ is invertible, and similar to the argument in the proof of Theorem \ref{thm:exp_params_nowhere_dense}, on $\{\tilde{b} > 0\}$ we have the identity $\tilde{t}(x_V) = (W^{-1})^\top h(x_V)$ with $h_i(x_V) := \log p_{\theta^i}(x_V) - \log p_{\theta^0}(x_V) + C_i$ for suitable constants $C_i$. The function that equals $(W^{-1})^\top h$ on $\{\tilde{b}>0\}$ and $0$ elsewhere is a minimal statistic (as it agrees with $\tilde{t}$, $\tilde{b}\mu$-a.s.) which is continuous at every $x_V\notin N_V$ with $\tilde{b}(x_V) > 0$, and hence continuous $\PP_\theta$-a.e., for every $\theta\in\Theta$. Hence, by \cite{barndorff-nielsen2014information}, Theorem 8.3, we have that $\tilde{\eta}(\Theta)$ and $M_G^{\mathrm{exp}}$ are homeomorphic, where $M_G^{\mathrm{exp}}$ is equipped with the weak topology.\footnote{Formally, \cite{barndorff-nielsen2014information}, Theorem 8.3(ii) assumes a minimal statistic that is continuous everywhere. Its proof only uses that $\PP_{\theta_n}\wto \PP_\theta$ implies $\tilde{t}_*\PP_{\theta_n}\wto \tilde{t}_*\PP_\theta$, which by the continuous mapping theorem holds as soon as the discontinuity points of the statistic form a $\PP_\theta$-null set, as is the case here.} Note that Barndorff-Nielsen actually shows that the convergence $\eta_n \to \eta$ in $\tilde{\eta}(\Theta)$ implies $\PP_{\eta_n} \tvto \PP_\eta$, so the weak topology and the total variation topology coincide on $M_G^{\mathrm{exp}}$.

	If there is a faithful parameter, then by Theorem \ref{thm:exp_params_nowhere_dense} the faithful parameters are dense in $\Theta$. From the proof of Theorem \ref{thm:exp_params_nowhere_dense} we have that the natural parameter $\theta\mapsto \tilde{\eta}(\theta)$ is analytic and hence continuous. The faithful parameters are dense in $\tilde{\eta}(\Theta)$, hence also in $M_G^{\mathrm{exp}}$, that is, $F_G^{\mathrm{exp}}$ is dense in $M_G^{\mathrm{exp}}$. By Theorem \ref{thm:ci_is_closed} conditional independence is closed in $M_G^{\mathrm{exp}}$, and hence $F_G^{\mathrm{exp}}$ is open in $M_G^{\mathrm{exp}}$.
\end{proof}

\section{Nonparametric conditional density models}\label{sec:equicont}
Due to their parametric nature, exponential families might not be flexible enough for certain statistical applications. In this section, we consider distributions on the separable complete metric spaces $\Xcal_V = \prod_{v\in V}\Xcal_v$ that have uniformly continuous and uniformly bounded (conditional) densities with respect to a given locally finite measure $\mu_V := \bigotimes_{v\in V}\mu_v$. A \emph{modulus of continuity} $\omega$ is a nonzero function $\omega : [0,\infty) \to [0,\infty)$ with $\lim_{t \downarrow 0}\omega(t) = \omega(0)=0$, and a uniformly continuous function $f:\Xcal_V \to \RR$ \emph{admits the modulus $\omega$} if $|f(x) - f(x')| \leq \omega(d(x, x'))$ for all $x,x'$, where $d$ denotes the metric on $\Xcal_V$. Classes of uniformly continuous functions which all admit the same modulus of continuity are called \emph{uniformly equicontinuous}. Throughout we assume all moduli to be concave. Since the functions under consideration are bounded, this is without loss of generality: a function bounded by $K$, admitting a modulus $\omega$, also admits the least concave majorant of $\omega \wedge 2K$ as modulus.

\begin{definition}
	Given a DAG $G$ with vertices $V$, a measure $\mu_v$ for each $v\in V$, modulus of continuity $\omega$ and $K > 0$, the class of \emph{equicontinuous and bounded Bayesian networks $\BN_G^{\mathrm{eqb}}$ induced by the tuple $(K, \omega, \mu)$} has for each $v\in V$ a conditional density $(x_v, x_{\pa(v)})\mapsto p(x_v \given x_{\pa(v)})$ with respect to $\mu_v$ which is bounded by $K$ and admits modulus of continuity $\omega$.
\end{definition}

Given a tuple $(K, \omega, \mu)$, let $\Pcal^{\mathrm{eqb}}(\Xcal_V)\subseteq \Pcal(\Xcal_V)$ be the set of probability measures which have a density with respect to $\mu$ which are bounded by $K$ and admit the modulus of continuity $\omega$.
The following is essentially a reformulation of \cite{boos1985converse}, Lemma 1, extending the original result to more general sample spaces.
\begin{lemma}\label{thm:equicont_weak_tv_coincide}
	The weak topology and total variation topology coincide on $\Pcal^{\mathrm{eqb}}(\Xcal_V)$, which is closed in $\Pcal(\Xcal_V)$.
\end{lemma}
\begin{proof}
	Let $\PP_n \wto \PP$ weakly with $\PP_n\in \Pcal^{\mathrm{eqb}}(\Xcal_V)$ for all $n\in\NN$. Let $D \subseteq \Xcal_V$ be a countable dense subset. For every subsequence $n'$ we have that $p_{n'}(x_1)$ is bounded by $K$, by Bolzano-Weierstrass there is a  further subsequence $n^{(1)}$ such that $p_{n^{(1)}}(x_1)$ converges. At $x_2$ we thus have a further subsequence $n^{(2)}$ such that $p_{n^{(2)}}$ converges both at $x_1$ and $x_2$, so by iterating this, diagonalisation gives a subsequence $n''$ of $n'$ such that $p_{n''}$ converges on $D$, to some $p^* : D \to [0,K]$. One readily verifies that $p^*$ also admits modulus of continuity $\omega$, so it extends to a function $p^*:\Xcal_V \to [0,K]$ with modulus of continuity $\omega$, for which we have $p_{n''} \to p^*$ pointwise.

	By Fatou's Lemma we have $\int p^*\diff\mu \leq \lim_{n''}\int p_{n''}\diff\mu \leq 1$. Since $\PP_{n''} \wto \PP$, the sequence $(\PP_{n''})$ is uniformly tight by Prokhorov's theorem (\citealp{bogachev2007measurea}, Theorem 8.6.4). For any $\varepsilon > 0$, let $K_\varepsilon$ be a compact set with $\PP_{n''}(K_\varepsilon^c) \leq \varepsilon$ for all $n''$ in the further subsequence. Since $\mu$ is locally finite, we have $\mu(K_\varepsilon) < \infty$. Since $p_{n''} \to p^*$ and $\mu(K_\varepsilon) < \infty$, dominated convergence gives $\int p^* d\mu \geq \int_{K_\varepsilon} p^* d\mu = \lim_{n''} \int_{K_\varepsilon} p_{n''} d\mu \geq 1 - \varepsilon$, so $\int p^* d\mu \geq 1$. Hence $p^*$ integrates to 1, and so $\PP^*\in \Pcal^{\mathrm{eqb}}(\Xcal_V)$.

	By \cite{scheffe1947useful}, convergence of the densities implies weak convergence $\PP_{n''}\wto \PP^*$. The weak convergence $\PP_n\wto \PP$ also implies convergence of the subsequence $\PP_{n''} \wto \PP$, and thus $\PP = \PP^*$, so $\PP$ has density $p^*$ and $\PP\in \Pcal^{\mathrm{eqb}}(\Xcal_V)$. The pointwise convergence $p_{n''} \to p^*$ implies $d_{TV}(\PP_{n''}, \PP)\to 0$, again by Scheffé's Theorem. Since this establishes that every subsequence of $d_{TV}(\PP_n, \PP)$ has a further subsequence converging to 0, we have $d_{TV}(\PP_{n}, \PP)\to 0$ as well (otherwise there exists $\varepsilon > 0$ and a subsequence $n'$ with $d_{TV}(\PP_{n'}, \PP) > \varepsilon$, contradicting the existence of a convergent further subsequence).
\end{proof}

For a specific class of equicontinuous and bounded Bayesian networks $\BN_G^{\mathrm{eqb}}$, the induced set of observational distributions is given by $M_G^{\mathrm{eqb}} := D(\BN_G^{\mathrm{eqb}})$. For each model in $\BN_G^{\mathrm{eqb}}$, the corresponding joint density $p(x_V)$ in $M_G^{\mathrm{eqb}}$ admits the modulus of continuity $\omega' := |V| (1\vee K)^{|V|-1} \omega$ and bound $K' := 1\vee K^{|V|}$, so $M_G^{\mathrm{eqb}}$ (with parameters $(K, \omega, \mu)$) is contained in $\Pcal^{\mathrm{eqb}}(\Xcal_V)$ (with parameters $(K', \omega', \mu)$). Hence, the weak topology and total variation topology also coincide on $M_G^{\mathrm{eqb}}$.

\subsection{Typicality of faithfulness in the space of Bayesian networks}\label{sec:equicont_params}
Similar as in Section \ref{sec:bn_params}, we prove the typicality of the faithful Bayesian networks.
First, we obtain the result that if there is at least one faithful Bayesian network in $\BN_G^{\mathrm{eqb}}$, then faithfulness is typical with respect to the metric $d_{TV}^\circ$, which is closely related to the weak topology for this model class.
For each $v \in V$ and $x_{\pa(v)}$ the conditional $\PP(X_v \given x_{\pa(v)})$ lies in $\Pcal^{\mathrm{eqb}}(\Xcal_v)$ (with parameters $(K, \omega, \mu_v)$). By Lemma~\ref{thm:equicont_weak_tv_coincide}, the total variation and weak topologies coincide on each of these spaces of probability measures. In particular, convergence in $d_{TV}^\circ$ implies weak convergence of each Markov kernel, uniformly over its conditioning variable.

\begin{theorem}\label{thm:equicont_params_nowhere_dense}
	Given a DAG $G$ and class $\BN_G^{\mathrm{eqb}}$ induced by a tuple $(K,\omega,\mu)$, if there is at least one faithful model in $\BN_G^{\mathrm{eqb}}$, then the set of faithful Bayesian networks is open and dense in $(\BN_G^{\mathrm{eqb}}, d_{TV}^\circ)$.
\end{theorem}
\begin{proof}[Proof]
	It immediately follows from Theorem \ref{thm:ci_is_closed} and Lemma \ref{thm:dist_map_continuous} that the set of faithful Bayesian networks in $\BN_G^{\mathrm{eqb}}$ is open with respect to $d_{TV}^\circ$.
	One readily verifies that $\BN_G^{\mathrm{eqb}}$ is closed under taking mixtures, so similar as in Definition \ref{def:markov_interpolation}, we can take between any unfaithful model $m_0$ and faithful model $m_1$ the mixture $m_\lambda := (1-\lambda)m_0 + \lambda m_1$ for which we have $m_\lambda \in \BN_G^{\mathrm{eqb}}$. By applying Lemma \ref{thm:interpolation_dependence} (whose proof applies verbatim with $\mu_V$ as dominating measure in place of $\QQ = \PP_{1/2}$, since all kernels of $m_0$ and $m_1$ have densities with respect to $\mu_v$) to each $d$-connection in $G$ we have that the tail of the sequence $m_{1/n}$ is faithful, and by a similar derivation as in Theorem \ref{thm:bn_params_nowhere_dense} it converges to $m_0$, hence the faithful Bayesian networks are dense.
\end{proof}

For a specific model class, the question remains whether it contains a faithful Bayesian network. For real sample spaces and densities with respect to Lebesgue measure, this is confirmed by leveraging Corollary \ref{thm:spirtes_nowhere_dense}.
\begin{lemma}\label{thm:d_weak_completeness_gauss}
	Given a DAG $G$, let $\Xcal_V = \RR^{|V|}$, and let $\BN_G^{\mathrm{eqb}}$ be induced by a tuple $(K, \omega, \mu)$ with $\mu$ the Lebesgue measure. There is a faithful model $m \in \BN_G^{\mathrm{eqb}}$.
\end{lemma}
\begin{proof}
	Fix any $t_0 > 0$ and set $c := \omega(t_0)/t_0$, which is strictly positive since $\omega$ is nonzero, concave and vanishes at $0$. Any conditional density $p$ that is $c$-Lipschitz in $(x_v, x_{\pa(v)})$ and bounded by $K \wedge \omega(t_0)$ admits the modulus $\omega$ and the bound $K$.
	As function of $(x_v, x_{\pa(v)})$, the conditional density $p_\theta(x_v \given x_{\pa(v)})$ has supremum $S(\theta_v) := 1/\sqrt{2\pi \sigma_v^2}$ and Lipschitz constant $L(\theta_v) := \sqrt{1+\|\beta_v\|_2^2} / (\sqrt{2\pi e} \sigma_v^2)$, which are both continuous in $\theta_v$. Hence the set
	\begin{equation*}
		O := \left\{\theta \in \Theta_{\Ncal} : L(\theta_v) < c \text{ and } S(\theta_v) < K \wedge \omega(t_0) \text{ for all } v \in V\right\}
	\end{equation*}
	is open, and it is non-empty: taking $\beta_v = 0$ and $\sigma_v^2$ large enough makes $L(\theta_v)$ and $S(\theta_v)$ simultaneously arbitrarily small for every $v\in V$. By Corollary \ref{thm:spirtes_nowhere_dense}, the faithful parameters are dense in $\Theta_{\Ncal}$, so there is a faithful parameter $\theta^* \in O$, and the corresponding model lies in $\BN_G^{\mathrm{eqb}}$.
\end{proof}

\begin{corollary}\label{thm:equicont_params_lebesgue_nowhere_dense}
	Under the assumptions of Lemma \ref{thm:d_weak_completeness_gauss}, the set of faithful Bayesian networks is open and dense in $(\BN_G^{\mathrm{eqb}}, d_{TV}^\circ)$.
\end{corollary}

\subsection{Typicality of faithfulness in the space of observational distributions}\label{sec:equicont_dists}
Finally, we also show that the observational distributions induced by the model class $\BN_G^{\mathrm{eqb}}$ are typically faithful. To that end, recall the definition of the set of observational distributions $M_G^{\mathrm{eqb}}$ of $\BN_G^{\mathrm{eqb}}$ induced by a triple $(K, \omega, \mu)$, and define the faithful distributions as $F_G^{\mathrm{eqb}} := M_G^{\mathrm{eqb}}\cap F_G$.

\begin{theorem}\label{thm:equicont_dist_nowhere_dense}
	Given a DAG $G$ and class $\BN_G^{\mathrm{eqb}}$ induced by a tuple $(K,\omega,\mu)$, if there is at least one faithful model in $\BN_G^{\mathrm{eqb}}$, then the set of faithful distributions $F_G^{\mathrm{eqb}}$ is open and dense in $M_G^{\mathrm{eqb}}$, both in the weak topology and in the total variation metric.
\end{theorem}
\begin{proof}
	If there is a faithful model in $\BN_G^{\mathrm{eqb}}$, then the set of faithful models is dense in $(\BN_G^{\mathrm{eqb}}, d_{TV}^\circ)$. By Lemma \ref{thm:dist_map_continuous} the distribution map $D:(\BN_G^{\mathrm{eqb}}, d_{TV}^\circ) \to (M_G^{\mathrm{eqb}}, d_{TV})$ is continuous, so then the set $F_G^{\mathrm{eqb}}$ is dense in $M_G^{\mathrm{eqb}}$. It is open by Theorem \ref{thm:ci_is_closed}. By Lemma \ref{thm:equicont_weak_tv_coincide} the weak topology and the total variation topology coincide on $M_G^{\mathrm{eqb}}$.
\end{proof}

\section{Bayesian networks with latent variables}\label{sec:bayesnet_latent}
The assumption that all variables in the Bayesian network must be observed is often too restrictive in practice. When certain variables remain unobserved, a suitable modelling class is that of Bayesian networks with observed variables $V$ and latent variables $W$.

Given a DAG $G$ over $V\cup W$, the \emph{latent projection} of $G$ onto $V$ is the \emph{Acyclic Directed Mixed Graph} (ADMG) $G_V$ with vertices $V$, directed edges $a\to b$ if there is a path $a \to w_1 \to ... \to w_n \to b$ in $G$ with $w_i\in W$ for all $i=1, ..., n$ (if any), and bi-directed edges $a\tot b$ if there is a bifurcation $a \ot w_1 \ot ... \ot w_k \to ... \to w_n \to b$ in $G$ with $w_i\in W$ for all $i=1, ..., n$ \citep{verma1993graphical}. An example of a DAG $G$ and its latent projection $G_V$ is given in Figure \ref{fig:latent}.
\begin{figure}[!htb]
	\begin{subfigure}{.45\linewidth}
		\centering
		\begin{tikzpicture}[scale=1,xscale=1.3,yscale=1]
			\node[var] (A) at (0,0) {$A$};
			\node[var] (B) at (1,0) {$B$};
			\node[var] (L1) at (1.5,1) {$L_1$};
			\node[var] (L2) at (2,0) {$L_2$};
			\node[var] (C) at (3,0) {$C$};
			\draw[arr] (A) to (B);
			\draw[arr] (B) to (L2);
			\draw[arr] (L2) to (C);
			\draw[arr] (L1) to (A);
			\draw[arr] (L1) to (B);
			\draw[arr] (L1) to (C);
			\draw[arr] (L1) to (L2);
		\end{tikzpicture}
		\caption{DAG $G$}
		\label{fig:latent_1}
	\end{subfigure}
	\begin{subfigure}{.45\linewidth}
		\centering
		\begin{tikzpicture}[scale=1,xscale=1.3,yscale=1]
			\node[var] (A) at (0,0) {$A$};
			\node[var] (B) at (1.5,0) {$B$};
			\node[var] (C) at (3,0) {$C$};
			\draw[arr] (A) to (B);
			\draw[arr] (B) to (C);
			\draw[biarr,bend left] (A) to (B);
			\draw[biarr,bend left] (B) to (C);
			\draw[biarr,bend right] (A) to (C);
		\end{tikzpicture}
		\caption{Latent projection $G_V$}
		\label{fig:latent_2}
	\end{subfigure}
	\caption{DAG $G$ and latent projection $G_V$ onto $V := \{A, B, C\}$.}
	\label{fig:latent}
\end{figure}

The definition of $d$-separation for ADMGs (also known as \emph{$m$-separation} \citep{richardson2003markov}) employs an extended notion of a collider: given ADMG $G_V$ with path $\pi=a\sus...\sus b$, a \emph{collider} is a vertex $v$ with $\to v \ot$, $\tot v \ot$, $\to v \tot$ or $\tot v \tot$ in $\pi$. As for DAGs, sets of vertices $A$ and $B$ are \emph{$d$-separated} given $C$ in ADMG $G$, written $A\perp^d_G B \given C$, if for every path $\pi = a \sus ... \sus b$ between $a\in A$ and $b \in B$, there is a collider in $\pi$ that is not an ancestor of $C$, or there is a non-collider in $\pi$ in $C$.
The independence models of $G$ and $G_V$ with respect to $V$ are equal: for any $A, B, C\subseteq V$ we have $A\perp_{G}^d B \given C$ if and only if $A\perp_{G_V}^d B \given C$ \citep{verma1993graphical}. As a corollary the Markov property (\ref{eqn:mp}) also holds for the latent projection $G_V$ of Bayesian networks with latent variables.

The question that we consider is whether (parameters of) Bayesian networks with latent variables are typically faithful to their latent projection.
Write $U_{G_V}, F_{G_V}$ for the distributions in $M_G \subseteq \Pcal(\Xcal_{V\cup W})$ that are unfaithful and faithful with respect to the ADMG $G_V$ respectively.
The core observation for extending results of Sections \ref{sec:bn}, \ref{sec:exp} and \ref{sec:equicont} from DAGs to ADMGs is the following:
\begin{lemma}\label{thm:admg_unfaithful}
	Given DAG $G$ with vertices $V\cup W$ and its latent projection $G_V$ onto $V$, any distribution over $V\cup W$ that is unfaithful with respect to $G_V$ is also unfaithful with respect to $G$.
\end{lemma}
\begin{proof}
	The latent projection preserves $d$-separations, so the result follows immediately from the expression for the set of unfaithful distributions:
	\begin{equation*}
		\bigcup_{A\nperp^d_{G} B\given C} \{\PP \in M_G : X_A\Indep_\PP X_B \given X_C \}.
	\end{equation*}
	For $U_{G_V}$ the union ranges over subsets $A, B, C \subseteq V$ and for $U_{G}$ the union ranges over $A, B, C \subseteq V\cup W$ (with a $d$-connection), hence we get $U_{G_V}\subseteq U_G$.
\end{proof}

Now, the preceding results are straightforwardly extended to conclude for unconstrained Bayesian networks with latent variables, conditional exponential family Bayesian networks with latent variables, and equicontinuous and bounded Bayesian networks with latent variables that faithfulness (which is only required to hold with respect to the latent projection) is open and dense.
The extensions of the results for typicality in the space of observational distributions (Theorems \ref{thm:bn_dist_nowhere_dense}, \ref{thm:exp_dist_nowhere_dense} and \ref{thm:equicont_dist_nowhere_dense}) immediately follow from Lemma \ref{thm:admg_unfaithful}. Considering the distribution map $D: \BN_G \to M_G$, the inclusion $D^{-1}(U_{G_V}) \subseteq D^{-1}(U_{G})$ (which follows from Lemma~\ref{thm:admg_unfaithful}) implies that $D^{-1}(F_{G_V}) \supseteq D^{-1}(F_{G})$ is dense. Openness follows from the same argument as in the original theorems: $U_{G_V}$ is a finite union of conditional independence sets which are closed in the total variation topology by Theorem~\ref{thm:ci_is_closed}, and $D$ is continuous. This gives the extensions of the results for the Bayesian networks themselves (Theorems \ref{thm:bn_params_nowhere_dense} and \ref{thm:equicont_params_nowhere_dense} and Corollary \ref{thm:equicont_params_lebesgue_nowhere_dense}). The extensions of the results for Euclidean parameters (Theorem~\ref{thm:exp_params_nowhere_dense} and Corollaries \ref{thm:spirtes_nowhere_dense} and \ref{thm:meek_nowhere_dense}) follow by the same reasoning applied to $T = D \circ \varphi$.

\section{Discussion}\label{sec:discussion}
In this work, we have established several results concerning the typicality of the faithfulness property, showing in various settings that Bayesian networks are indeed typically faithful in both a topological and a measure-theoretic sense. We have shown:

\begin{itemize}
	\item For unconstrained nonparametric Bayesian networks:
	\begin{itemize}
		\item the faithful distributions are open and dense with respect to the total variation metric (Thm.\ \ref{thm:bn_dist_nowhere_dense});
		\item the faithful Bayesian networks are open and dense with respect to the metric $d_{TV}^\circ$ (Thm.\ \ref{thm:bn_params_nowhere_dense}).
	\end{itemize}
	\item For sufficiently regular conditional exponential family Bayesian networks:
	\begin{itemize}
		\item if there exists a faithful parameter, then the faithful parameters are open and dense with respect to the Euclidean topology and the unfaithful parameters have Lebesgue measure zero (Thm.\ \ref{thm:exp_params_nowhere_dense});
		\item if there exists a faithful parameter, then the faithful distributions are open and dense with respect to the weak topology and the total variation metric (Thm.\ \ref{thm:exp_dist_nowhere_dense});
		\item for linear Gaussian and discrete Bayesian networks, there is a faithful parameter, so faithfulness is typical in these ways (Cor.\ \ref{thm:spirtes_nowhere_dense} and \ref{thm:meek_nowhere_dense}).
	\end{itemize}
	\item For equicontinuous and bounded Bayesian networks:
	\begin{itemize}
		\item if there exists a faithful model, then the faithful models are open and dense with respect to the metric $d_{TV}^\circ$ (Thm.\ \ref{thm:equicont_params_nowhere_dense});
		\item if there exists a faithful model, then the faithful distributions are open and dense in the weak topology and the total variation metric (Thm.\ \ref{thm:equicont_dist_nowhere_dense});
		\item for real sample spaces and densities with respect to Lebesgue measure, there exists a faithful parameter, so faithful Bayesian networks c.q.\ distributions are open and dense (Cor.\ \ref{thm:equicont_params_lebesgue_nowhere_dense}).
	\end{itemize}
\end{itemize}

This collection of results naturally raises the question: what is the relative value of these results, with their different notions of typicality?
The two main notions that we use --- topological and measure-theoretic --- do not necessarily coincide.
For example, the Smith-Volterra-Cantor set is a nowhere dense subset of $[0,1]$ that has Lebesgue measure 1/2.
In general, \emph{every} subset of $\RR$ is the disjoint union of a meager set and a Lebesgue null set (\citealp{oxtoby1980measure}, Theorem 1.6): a set that is small in one sense may be large in the other sense.

The measure-theoretic approach, which identifies `atypical' sets with those of (Lebesgue) measure zero, is for example useful when one samples a Bayesian network, e.g.\ for performing simulations. If the unfaithful parameters have measure zero, the probability of encountering one when sampling from a distribution with a density is zero. Note that this $\sigma$-ideal depends on the choice of $\sigma$-algebra and measure. A restriction is that this approach is challenging in the nonparametric case, as no canonical analogue of Lebesgue measure exists in infinite-dimensional spaces. For infinite-dimensional locally compact topological groups the Haar measure is a natural choice, but the space of (observational distributions of) Bayesian networks is generally not locally compact with respect to the weak or total variation topology, nor a group. \cite{hunt1992prevalence} introduce the notion of \emph{shy sets} as analogue of Lebesgue-null sets in arbitrary linear metric spaces. The space of probability distributions is not linear so this concept is not directly applicable in our case.

The topological approach identifies `typical' sets with those that are either open and dense, or complements of nowhere dense sets, or complements of meager sets. Clearly, such notions of typicality depend on the choice of the topology. The total variation topology is convenient, since here, conditional independence is a closed property. The weak topology is a natural choice as it is closely related to testability \citep{dembo1994topological, genin2017topology,boeken2026topological}.
However, obtaining results in the weak topology requires finding regularity conditions such that conditional independence is closed, which does not hold in general. Finding such regularity conditions is challenging. One approach is to find regularity conditions such that the weak topology and total variation topology coincide. To this end, the employed conditions by \cite{barndorff-nielsen2014information} and \cite{boos1985converse} for exponential families and equicontinuous and bounded densities are convenient.
\cite{boeken2026topological} deviate from this approach of equating the total variation topology and weak topology, as they provide alternative uniform continuity conditions on the Markov kernels $x_B\mapsto \PP(X_A\given x_B)$, for which they directly show that conditional independence is closed in the weak topology. However, these regularity conditions fail to combine nicely with the interpolation of Bayesian networks that we consider in our proof technique for showing that faithful Bayesian networks are dense, hence we do not follow that approach in the current work.
We have introduced the metric $d_{TV}^\circ$ on the space of Bayesian networks, which measures the worst-case distance between corresponding Markov kernels. This is a natural choice when viewing Bayesian networks as causal models, where the Markov kernels represent mechanisms that are defined for all parent values.
Alternative metrics could be considered, for instance by replacing the supremum over parent values with an average weighted by some reference measure, which would yield a weaker topology and potentially stronger typicality results. However, such a metric would depend on the choice of reference measure, and would not distinguish between Bayesian networks that differ only on parent values with zero reference measure --- an undesirable property for causal models.

\subsection{Implications for constraint-based causal discovery}
The topological properties of conditional independence have two important implications for constraint-based causal discovery:
\begin{itemize}
	\item when conditional independence is closed in the weak topology there exists a consistent test, which can be used to make any sound constraint-based causal discovery algorithm consistent, and
	\item the set of faithful Bayesian networks is open and dense, so any causal discovery algorithm that is consistent under the faithfulness assumption has a topologically large `domain of consistency'.
\end{itemize}

For the first statement, consider any given set of probability measures $\Pcal$ on $\Xcal_A\times \Xcal_B\times \Xcal_C$, and let $H_0 = \{\PP \in \Pcal : X_A\Indep_\PP X_B\given X_C\}$ and $H_1 = \{\PP \in \Pcal : X_A\nIndep_\PP X_B\given X_C\}$.
Under various regularity conditions on $\Pcal$, \cite{dembo1994topological,ermakov2017consistent,genin2017topology} and \cite{boeken2026topological} give topological characterisations for the (uniformly) consistent testability of any pair of statistical hypotheses. Importantly, \cite{genin2017topology} (Theorem 4.1) show for distributions that have a density with respect to a common dominating measure, that if a null-hypothesis $H_0$ is closed in the weak topology, then there exists a consistent test, i.e.\ for every $n\in\NN$ there exists a map $\varphi_n : \Xcal_V^n \to \{0,1\}$ such that $\lim_{n\to\infty}\PP^n(\varphi_n=i) = 1$ if and only if $\PP \in H_i$.\footnote{Actually, \cite{genin2017topology} show more: that this test is consistent and has type I error control. We ignore this property because it is unclear how type I error control translates to error control for causal discovery algorithms. This suggests that we are restricting the model classes more than necessary. Indeed, the existence of a merely consistent test (without requiring type I error control) is equivalent to the null and alternative hypotheses being $F_\sigma$ in the weak topology (\citealp{genin2017topology}, Theorem 4.3). This can be shown to hold for conditional (in)dependence in countable unions of the classes considered in Theorem \ref{thm:ci_is_closed_weak}, and for which analogues of Theorem \ref{thm:consistent_discovery} and Corollary \ref{thm:consistent_discovery_maximal} hold with \emph{comeager} domains of consistency.}
For classes of distributions where the weak topology and the total variation topology coincide, it follows from Theorem \ref{thm:ci_is_closed} that conditional independence is consistently testable.

Since causal discovery algorithms are agnostic of the underlying causal graph, we need to consider the space of Bayesian networks and induced observational distributions of Bayesian networks with all possible causal graphs over a fixed set of variables. To formalise this, let $V$ be a finite index set, let $\Xcal_V = \prod_{v\in V}\Xcal_v$ be a product of separable complete metric spaces, let $\mu$ be a dominating measure on $\Xcal_V$, and finally let $\BN^{\mathrm{eqb}} := \bigcup_{G}\{G\} \times \BN_G^{\mathrm{eqb}}$ be the set of all Bayesian networks, where the union is taken over all DAGs $G$ with vertices $V$. Similar to \cite{lin2020learning}, equip this space with the metric $\delta + d_{TV}^\circ$, where $\delta$ is the discrete metric on the space of DAGs with vertices $V$. This way we can extend the result about consistent testability without relying on a specific graph $G$.
\begin{theorem}\label{thm:ci_is_closed_weak}
	Conditional independence is consistently testable in the set of observational distributions induced by $\BN^{\mathrm{eqb}}$.
\end{theorem}
Since the weak topology and total variation topology coincide for the conditional exponential family Bayesian networks of Theorem \ref{thm:exp_dist_nowhere_dense}, a similar result can be derived for this model class.

For the second statement, it follows from Theorem \ref{thm:equicont_params_nowhere_dense} that if for every $G$ the set $\BN_G^{\mathrm{eqb}}$ contains a faithful Bayesian network, then the set of all faithful Bayesian networks is open and dense in $\BN^{\mathrm{eqb}}$.
Hence, when combined with a conditional independence test as in Theorem \ref{thm:ci_is_closed_weak}, any constraint-based causal discovery algorithm that is sound under the assumption of faithfulness (i.e.\ arrives at the correct conclusion when using a conditional independence oracle) is consistent on a topologically large set among all possible Bayesian networks. This in particular holds for the \emph{PC-algorithm} and the \emph{FCI-algorithm} \citep{spirtes1993causation}.

\begin{theorem}\label{thm:consistent_discovery}
	Given a set of vertices $V$, if for each DAG $G$ with vertices $V$ there exists a faithful Bayesian network in $\BN_G^{\mathrm{eqb}}$, then every constraint-based causal discovery algorithm that is sound under the assumption of faithfulness is consistent on an open and dense domain in $\BN^{\mathrm{eqb}}$.
\end{theorem}

Obtaining these results for arbitrary Bayesian networks without any regularity assumptions on the distributions is impossible. Although conditional independence is closed with respect to the total variation metric, this does not imply that it is consistently testable. Indeed, \cite{shah2020hardness} and \cite{neykov2021minimax} prove that without imposing regularity conditions, conditional independence is not consistently testable with type I error control; \cite{boeken2026topological} prove that it is not consistently testable from finite-precision data.

Faithfulness is sufficient, but not \emph{necessary} for consistent constraint-based causal discovery. Weaker conditions have been proposed, like adjacency faithfulness \citep{spirtes1993causation}, P-minimality \citep{pearl2009causality}, SGS-minimality \citep{spirtes1993causation} and u-minimality \cite{zhang2013comparison,lin2020learning}. \cite{zhang2013comparison} has shown that SGS-minimality is strictly weaker than P-minimality, and that u-minimality sits strictly between faithfulness and P-minimality. Further, u-minimality is sufficient for constraint-based causal discovery, and P-minimality is necessary \citep{lin2020learning}. Since these sets of Bayesian networks include all faithful Bayesian networks, they are the complement of a nowhere dense set if the faithful Bayesian networks are open and dense. In other words, the typicality of faithful Bayesian networks implies the typicality of even larger classes of Bayesian networks, like the u-minimal Bayesian networks.

\cite{lin2020learning} show that there exist constraint-based causal discovery algorithms that are consistent on a \emph{maximal} domain. In particular, they show that there exists a causal discovery algorithm $\hat{H}$ that is consistent for all u-minimal Bayesian networks, and there is no algorithm whose domain of consistency strictly contains that of $\hat H$, so the domain of consistency is maximal. Combined with Theorem \ref{thm:ci_is_closed_weak} this gives the following result.

\begin{corollary}\label{thm:consistent_discovery_maximal}
	Given a set of vertices $V$, if for each DAG $G$ with vertices $V$ there exists a faithful Bayesian network in $\BN_G^{\mathrm{eqb}}$, then there exists a causal discovery algorithm that is consistent on a maximal domain in $\BN^{\mathrm{eqb}}$ whose complement is nowhere dense.
\end{corollary}

Faithful distributions can possess extremely weak dependencies that are hard to test for. For linear Gaussian networks, \cite{zhang2002strong} consider \emph{strong faithfulness}, i.e.\ the condition that every $d$-connection in the graph has in the distribution a corresponding conditional dependence (partial correlation in their case) with some minimal strength. The set of parameters which exhibit a weak dependency is known to be of strict positive measure \citep{uhler2013geometry}, so strong faithfulness is not typical in the measure-theoretic sense. That every conditional dependence has a minimal strength implies the existence of a uniformly consistent test, and hence of uniformly consistent causal discovery algorithms. Namely, under tightness conditions, \cite{boeken2026topological} show that uniformly consistent testability is implied by separation of the hypotheses in the \emph{bounded Lipschitz metric $d_{BL}$} for the weak topology (i.e.\ the condition that $d_{BL}(H_0, H_1) > 0$) and that the nonparametric version of the minimal strength of the conditional dependence (i.e.\ that for some $\varepsilon>0$ the alternative hypothesis $H_1$ consists of measures with $d_{BL}(\PP(X\given Z)\otimes \PP(Y, Z), \PP(X, Y, Z)) > \varepsilon$) satisfies this separation property under certain regularity conditions. We conjecture that under these regularity assumptions and this generalised notion of strong faithfulness, uniformly consistent constraint-based causal discovery can be achieved.

\subsection{Concluding remarks}

A number of follow-up questions remain. First, although there is no canonical measure on the space of Bayesian networks, one might construct specific ones. For example, conditional optional Pólya trees \citep{ma2017recursive} provide a flexible class of random conditional measures. These can straightforwardly be extended to random Bayesian networks with a given graph. It would be of interest to know whether faithfulness has full measure.

\cite{sadeghi2017faithfulness} characterises faithfulness in terms of the properties intersectionality, compositionality, singleton-transitivity, and ordered downward- and upward-stability. Our various typicality results for faithfulness directly imply corresponding typicality results for those constituent properties. Sadeghi equates very similar properties to faithfulness of distributions to more general classes of graphs (e.g.\ chain graphs and ancestral graphs). It would be interesting to see whether this approach leads to typicality results for faithfulness of distributions with respect to these more general classes of graphs.

Finally, our work focusses on acyclic causal models. However, constraint-based causal discovery algorithms also exist for certain classes of uniquely solvable cyclic models, like \emph{simple SCMs} \citep{bongers2021foundations}. These causal discovery algorithms either rely on versions of faithfulness defined in terms of $d$-separations or $\sigma$-separation, see e.g.\ \cite{richardson1996discovery}, \cite{strobl2019constraintbased} and \cite{mooij2020constraintbased}.
Our proof techniques for acyclic models do not immediately transfer to cyclic models.
For example, it is not clear whether the interpolation of two simple SCMs is again a simple SCM. It therefore remains an open question whether simple SCMs are typically faithful.

\section{Acknowledgements}
This research was supported by Booking.com. Part of this work was carried out while Philip Boeken was at the Korteweg–de Vries Institute for Mathematics, University of Amsterdam. We thank Konstantin Genin for helpful remarks and suggestions. In preparing the final version of this manuscript, we used the large language model Claude (Anthropic) to check the proofs for errors and to suggest corrections; all resulting changes were verified by the authors, who remain fully responsible for the correctness of the manuscript.

\bibliographystyle{apalike}
\bibliography{refs}

@book{aliprantis2006infinite,
  title      = {Infinite Dimensional Analysis: A Hitchhiker's Guide},
  author     = {Aliprantis, Charalambos D. and Border, Kim C.},
  year       = 2006,
  edition    = {3rd [rev. and enl.] ed},
  publisher  = {Springer},
  address    = {Berlin ; New York},
  isbn       = {978-3-540-29586-0},
  lccn       = {QA320 .A45 2006}
}

@incollection{bareinboim2022pearls,
  title     = {On {{Pearl}}'s {{Hierarchy}} and the {{Foundations}} of {{Causal Inference}}},
  booktitle = {Probabilistic and {{Causal Inference}}: {{The Works}} of {{Judea Pearl}}},
  author    = {Bareinboim, Elias and Correa, Juan D. and Ibeling, Duligur and Icard, Thomas},
  year      = {2022},
  month     = {mar},
  volume    = {36},
  pages     = {507--556},
  publisher = {Association for Computing Machinery},
  address   = {New York, NY, USA},
  url       = {https://doi.org/10.1145/3501714.3501743},
  isbn      = {978-1-4503-9586-1}
}

@book{barndorff-nielsen2014information,
  title     = {Information and Exponential Families: In Statistical Theory},
  author    = {{Barndorff-Nielsen}, Ole E.},
  year      = 2014,
  publisher = {John Wiley \& Sons},
  address   = {Chichester [U.K.] New York},
  doi       = {10.1002/9781118857281},
  isbn      = {978-1-118-85750-2 978-1-118-85737-3 978-1-306-77257-0 978-1-118-85728-1},
  langid    = {english}
}

@misc{boeken2026topological,
  title = {Topological {{Criteria}} for {{Hypothesis Testing}} with {{Finite-Precision Measurements}}},
  author = {Boeken, Philip and Skapinakis, Eduardo and Genin, Konstantin and Mooij, Joris M.},
  year = 2026,
  month = jan,
  number = {arXiv:2601.13946},
  eprint = {2601.13946},
  primaryclass = {math},
  publisher = {arXiv},
  doi = {10.48550/arXiv.2601.13946},
  url = {http://arxiv.org/abs/2601.13946},
  archiveprefix = {arXiv}
}

@book{bogachev2007measurea,
  title     = {Measure Theory Vol. {{II}}},
  author    = {Bogachev, V. I.},
  year      = 2007,
  publisher = {Springer},
  address   = {Berlin ; New York},
  isbn      = {978-3-540-34513-8},
  langid    = {english},
  lccn      = {QA312 .B645 2007}
}

@article{bongers2021foundations,
  title     = {Foundations of structural causal models with cycles and latent variables},
  author    = {Bongers, Stephan and Forr{\'e}, Patrick and Peters, Jonas and Mooij, Joris M},
  journal   = {The Annals of Statistics},
  volume    = {49},
  number    = {5},
  pages     = {2885--2915},
  year      = {2021},
  publisher = {Institute of Mathematical Statistics}
}

@article{boos1985converse,
  title   = {A {{Converse}} to {{Scheffe}}'s {{Theorem}}},
  author  = {Boos, Dennis D.},
  year    = 1985,
  month   = mar,
  journal = {The Annals of Statistics},
  volume  = {13},
  number  = {1},
  issn    = {0090-5364},
  doi     = {10.1214/aos/1176346604},
  url     = {https://projecteuclid.org/journals/annals-of-statistics/volume-13/issue-1/A-Converse-to-Scheffes-Theorem/10.1214/aos/1176346604.full}
}

@article{dembo1994topological,
  title   = {A {{Topological Criterion}} for {{Hypothesis Testing}}},
  author  = {Dembo, Amir and Peres, Yuval},
  year    = 1994,
  month   = mar,
  journal = {The Annals of Statistics},
  volume  = {22},
  number  = {1},
  issn    = {0090-5364},
  doi     = {10.1214/aos/1176325360},
  url     = {https://projecteuclid.org/journals/annals-of-statistics/volume-22/issue-1/A-Topological-Criterion-for-Hypothesis-Testing/10.1214/aos/1176325360.full}
}

@article{ermakov2017consistent,
  title   = {On {{Consistent Hypothesis Testing}}},
  author  = {Ermakov, M.},
  year    = 2017,
  month   = sep,
  journal = {Journal of Mathematical Sciences},
  volume  = {225},
  number  = {5},
  pages   = {751--769},
  issn    = {1573-8795},
  doi     = {10.1007/s10958-017-3491-4},
  url     = {https://doi.org/10.1007/s10958-017-3491-4},
  langid  = {english}
}

@article{feigin1981conditional,
  title     = {Conditional {{Exponential Families}} and a {{Representation Theorem}} for {{Asympotic Inference}}},
  author    = {Feigin, Paul D.},
  year      = {1981},
  month     = may,
  journal   = {The Annals of Statistics},
  volume    = {9},
  number    = {3},
  pages     = {597--603},
  publisher = {Institute of Mathematical Statistics},
  issn      = {0090-5364, 2168-8966},
  doi       = {10.1214/aos/1176345463},
  url       = {https://projecteuclid.org/journals/annals-of-statistics/volume-9/issue-3/Conditional-Exponential-Families-and-a-Representation-Theorem-for-Asympotic-Inference/10.1214/aos/1176345463.full}
}

@article{geiger1990identifying,
  title   = {Identifying Independence in {B}ayesian Networks},
  author  = {Geiger, Dan and Verma, Thomas and Pearl, Judea},
  year    = {1990},
  month   = aug,
  journal = {Networks},
  volume  = {20},
  number  = {5},
  pages   = {507--534},
  issn    = {00283045, 10970037},
  doi     = {10.1002/net.3230200504},
  url     = {https://onlinelibrary.wiley.com/doi/10.1002/net.3230200504},
  langid  = {english}
}

@article{genin2017topology,
  title         = {The {{Topology}} of {{Statistical Verifiability}}},
  author        = {Genin, Konstantin and Kelly, Kevin T.},
  year          = 2017,
  month         = jul,
  journal       = {Electronic Proceedings in Theoretical Computer Science},
  volume        = {251},
  eprint        = {1707.09378},
  primaryclass  = {cs, math},
  pages         = {236--250},
  issn          = {2075-2180},
  doi           = {10.4204/EPTCS.251.17},
  url           = {http://arxiv.org/abs/1707.09378},
  archiveprefix = {arXiv}
}

@book{gunning1965analytic,
  title       = {Analytic {{Functions}} of {{Several Complex Variables}}},
  author      = {Gunning, Robert C. and Rossi, Hugo},
  year        = 1965,
  publisher   = {Prentice-Hall},
  googlebooks = {TR6bEAAAQBAJ},
  isbn        = {978-1-4704-7066-1},
  langid      = {english}
}

@article{hunt1992prevalence,
  title   = {Prevalence: A Translation-Invariant ``Almost Every'' on Infinite-Dimensional Spaces},
  author  = {Hunt, Brian R. and Sauer, Tim and Yorke, James A.},
  year    = 1992,
  journal = {Bulletin of the American Mathematical Society},
  volume  = {27},
  number  = {2},
  pages   = {217--238},
  issn    = {0273-0979, 1088-9485},
  doi     = {10.1090/S0273-0979-1992-00328-2},
  url     = {https://www.ams.org/bull/1992-27-02/S0273-0979-1992-00328-2/},
  langid  = {english}
}

@inproceedings{ibeling2021topological,
  title     = {A {{Topological Perspective}} on {{Causal Inference}}},
  booktitle = {Advances in {{Neural Information Processing Systems}}},
  author    = {Ibeling, Duligur and Icard, Thomas},
  year      = {2021},
  volume    = {34},
  pages     = {5608--5619},
  publisher = {Curran Associates, Inc.},
  url       = {https://proceedings.neurips.cc/paper/2021/hash/2c463dfdde588f3bfc60d53118c10d6b-Abstract.html}
}

@book{kechris2012classical,
  title     = {Classical descriptive set theory},
  author    = {Kechris, Alexander},
  year      = {1995},
  publisher = {Springer}
}

@article{lang1986note,
  title   = {A Note on the Measurability of Convex Sets},
  author  = {Lang, Robert},
  year    = 1986,
  month   = jul,
  journal = {Archiv der Mathematik},
  volume  = {47},
  number  = {1},
  pages   = {90--92},
  issn    = {1420-8938},
  doi     = {10.1007/BF01202504},
  url     = {https://doi.org/10.1007/BF01202504},
  langid  = {english}
}

@book{lauritzen1996graphical,
  title     = {Graphical models},
  author    = {Lauritzen, Steffen},
  year      = {1996},
  publisher = {Clarendon Press}
}

@article{lauritzen2024total,
  title    = {Total Variation Convergence Preserves Conditional Independence},
  author   = {Lauritzen, Steffen},
  year     = {2024},
  month    = {nov},
  journal  = {Statistics \& Probability Letters},
  volume   = {214},
  pages    = {110200},
  issn     = {0167-7152},
  doi      = {10.1016/j.spl.2024.110200},
  url      = {https://www.sciencedirect.com/science/article/pii/S016771522400169X}
}

@inproceedings{lin2020learning,
  title     = {On {{Learning Causal Structures}} from {{Non-Experimental Data}} without {{Any Faithfulness Assumption}}},
  booktitle = {Proceedings of the 31st {{International Conference}}  on {{Algorithmic Learning Theory}}},
  author    = {Lin, Hanti and Zhang, Jiji},
  year      = {2020},
  month     = jan,
  pages     = {554--582},
  publisher = {PMLR},
  issn      = {2640-3498},
  url       = {https://proceedings.mlr.press/v117/lin20a.html},
  langid    = {english}
}

@article{ma2017recursive,
  title   = {Recursive Partitioning and Multi-Scale Modeling on Conditional Densities},
  author  = {Ma, Li},
  year    = 2017,
  month   = jan,
  journal = {Electronic Journal of Statistics},
  volume  = {11},
  number  = {1},
  issn    = {1935-7524},
  doi     = {10.1214/17-EJS1254},
  url     = {https://projecteuclid.org/journals/electronic-journal-of-statistics/volume-11/issue-1/Recursive-partitioning-and-multi-scale-modeling-on-conditional-densities/10.1214/17-EJS1254.full},
  langid  = {english}
}

@inproceedings{meek1995strong,
  author    = {Meek, Christopher},
  title     = {Strong Completeness and Faithfulness in {B}ayesian Networks},
  year      = {1995},
  isbn      = {1558603859},
  publisher = {Morgan Kaufmann Publishers Inc.},
  address   = {San Francisco, CA, USA},
  booktitle = {Proceedings of the Eleventh Conference on Uncertainty in Artificial Intelligence},
  pages     = {411–418},
  numpages  = {8},
  location  = {Montr\'{e}al, Qu\'{e}, Canada},
  series    = {UAI’95}
}

@phdthesis{meek1998graphical,
  title  = {Graphical {{Models}}: {{Selecting}} Causal and Statistical Models},
  author = {Meek, Christopher},
  year   = {1998},
  langid = {english},
  school = {Carnegie Mellon University}
}

@article{mityagin2020zero,
  title    = {The {{Zero Set}} of a {{Real Analytic Function}}},
  author   = {Mityagin, B. S.},
  year     = 2020,
  month    = mar,
  journal  = {Mathematical Notes},
  volume   = {107},
  number   = {3},
  pages    = {529--530},
  issn     = {1573-8876},
  doi      = {10.1134/S0001434620030189},
  url      = {https://doi.org/10.1134/S0001434620030189},
  langid   = {english}
}

@inproceedings{mooij2020constraintbased,
  title     = {Constraint-{{Based Causal Discovery}} Using {{Partial Ancestral Graphs}} in the Presence of {{Cycles}}},
  booktitle = {UAI2020},
  author    = {Mooij, J.M. and Claassen, Tom},
  year      = {2020},
  pages     = {1159--1168},
  publisher = {PMLR},
  issn      = {2640-3498},
  url       = {https://proceedings.mlr.press/v124/m-mooij20a.html},
  langid    = {english}
}

@article{neykov2021minimax,
  title     = {Minimax Optimal Conditional Independence Testing},
  author    = {Neykov, Matey and Balakrishnan, Sivaraman and Wasserman, Larry},
  year      = {2021},
  month     = aug,
  journal   = {The Annals of Statistics},
  volume    = {49},
  number    = {4},
  pages     = {2151--2177},
  publisher = {{Institute of Mathematical Statistics}},
  issn      = {0090-5364, 2168-8966},
  doi       = {10.1214/20-AOS2030},
  url       = {https://projecteuclid.org/journals/annals-of-statistics/volume-49/issue-4/Minimax-optimal-conditional-independence-testing/10.1214/20-AOS2030.full}
}

@book{oxtoby1980measure,
  title     = {Measure and {{Category}}},
  author    = {Oxtoby, John C.},
  year      = {1980},
  series    = {Graduate {{Texts}} in {{Mathematics}}},
  edition   = {{S}econd},
  publisher = {{Springer New York}},
  address   = {{New York, NY}},
  doi       = {10.1007/978-1-4684-9339-9},
  url       = {http://link.springer.com/10.1007/978-1-4684-9339-9},
  isbn      = {978-1-4684-9341-2 978-1-4684-9339-9},
  langid    = {english}
}

@book{pearl1988probabilistic,
  title     = {Probabilistic Reasoning in Intelligent Systems: Networks of Plausible Inference},
  author    = {Pearl, Judea},
  year      = 1988,
  series    = {The {{Morgan Kaufmann}} Series in Representation and Reasoning},
  edition   = {Rev. 2. ed., transferred to digital printing},
  publisher = {Morgan Kaufmann},
  address   = {San Francisco, Calif},
  isbn      = {978-1-55860-479-7},
  langid    = {english}
}

@book{pearl2009causality,
  title       = {Causality},
  author      = {Pearl, Judea},
  year        = 2009,
  month       = sep,
  publisher   = {Cambridge University Press},
  googlebooks = {f4nuexsNVZIC},
  isbn        = {978-0-521-89560-6},
  langid      = {english}
}

@inproceedings{richardson1996discovery,
  title     = {A Discovery Algorithm for Directed Cyclic Graphs},
  booktitle = {Proceedings of the {{Twelfth}} International Conference on {{Uncertainty}} in Artificial Intelligence},
  author    = {Richardson, T.},
  year      = {1996},
  series    = {{{UAI}}'96},
  pages     = {454--461},
  publisher = {Morgan Kaufmann Publishers Inc.},
  address   = {San Francisco, CA, USA},
  isbn      = {978-1-55860-412-4}
}

@article{richardson2003markov,
  title      = {Markov {{Properties}} for {{Acyclic Directed Mixed Graphs}}},
  author     = {Richardson, Thomas},
  year       = {2003},
  journal    = {Scandinavian Journal of Statistics},
  volume     = {30},
  number     = {1},
  eprint     = {4616754},
  eprinttype = {jstor},
  pages      = {145--157},
  publisher  = {[Board of the Foundation of the Scandinavian Journal of Statistics, Wiley]},
  issn       = {0303-6898},
  url        = {https://www.jstor.org/stable/4616754}
}

@article{sadeghi2017faithfulness,
  title   = {Faithfulness of {{Probability Distributions}} and {{Graphs}}},
  author  = {Sadeghi, Kayvan},
  year    = 2017,
  journal = {Journal of Machine Learning Research},
  volume  = {18},
  number  = {148},
  pages   = {1--29},
  issn    = {1533-7928},
  url     = {http://jmlr.org/papers/v18/17-275.html}
}

@article{scheffe1947useful,
  title   = {A {{Useful Convergence Theorem}} for {{Probability Distributions}}},
  author  = {Scheff\'e, Henry},
  year    = {1947},
  month   = sep,
  journal = {The Annals of Mathematical Statistics},
  volume  = {18},
  number  = {3},
  pages   = {434--438},
  issn    = {0003-4851},
  doi     = {10.1214/aoms/1177730390},
  url     = {http://projecteuclid.org/euclid.aoms/1177730390},
  langid  = {english}
}

@article{shah2020hardness,
  title     = {The Hardness of Conditional Independence Testing and the Generalised Covariance Measure},
  author    = {Shah, Rajen D. and Peters, Jonas},
  year      = 2020,
  month     = jun,
  journal   = {The Annals of Statistics},
  volume    = {48},
  number    = {3},
  pages     = {1514--1538},
  publisher = {Institute of Mathematical Statistics},
  issn      = {0090-5364, 2168-8966},
  doi       = {10.1214/19-AOS1857},
  url       = {https://projecteuclid.org/journals/annals-of-statistics/volume-48/issue-3/The-hardness-of-conditional-independence-testing-and-the-generalised-covariance/10.1214/19-AOS1857.full}
}

@book{spirtes1993causation,
  title     = {Causation, {{Prediction}}, and {{Search}}},
  author    = {Spirtes, Peter and Glymour, Clark and Scheines, Richard},
  year      = {1993},
  series    = {Lecture {{Notes}} in {{Statistics}}},
  volume    = {81},
  publisher = {Springer},
  address   = {New York, NY},
  doi       = {10.1007/978-1-4612-2748-9},
  url       = {http://link.springer.com/10.1007/978-1-4612-2748-9},
  copyright = {http://www.springer.com/tdm},
  isbn      = {978-1-4612-7650-0 978-1-4612-2748-9}
}

@article{strobl2019constraintbased,
  title    = {A Constraint-Based Algorithm for Causal Discovery with Cycles, Latent Variables and Selection Bias},
  author   = {Strobl, Eric V.},
  year     = 2019,
  month    = jul,
  journal  = {International Journal of Data Science and Analytics},
  volume   = {8},
  number   = {1},
  pages    = {33--56},
  issn     = {2364-4168},
  doi      = {10.1007/s41060-018-0158-2},
  url      = {https://doi.org/10.1007/s41060-018-0158-2},
  langid   = {english}
}

@article{uhler2013geometry,
  title     = {Geometry of the Faithfulness Assumption in Causal Inference},
  author    = {Uhler, Caroline and Raskutti, Garvesh and B{\"u}hlmann, Peter and Yu, Bin},
  year      = {2013},
  month     = apr,
  journal   = {The Annals of Statistics},
  volume    = {41},
  number    = {2},
  pages     = {436--463},
  publisher = {Institute of Mathematical Statistics},
  issn      = {0090-5364, 2168-8966},
  doi       = {10.1214/12-AOS1080},
  url       = {https://projecteuclid.org/journals/annals-of-statistics/volume-41/issue-2/Geometry-of-the-faithfulness-assumption-in-causal-inference/10.1214/12-AOS1080.full}
}

@incollection{verma1990causal,
  title     = {Causal {{Networks}}: {{Semantics}} and {{Expressiveness}}},
  booktitle = {Machine {{Intelligence}} and {{Pattern Recognition}}},
  author    = {Verma, Thomas and Pearl, Judea},
  editor    = {Shachter, Ross D. and Levitt, Tod S. and Kanal, Laveen N. and Lemmer, John F.},
  year      = {1990},
  month     = jan,
  series    = {Uncertainty in {{Artificial Intelligence}}},
  volume    = {9},
  pages     = {69--76},
  publisher = {North-Holland},
  doi       = {10.1016/B978-0-444-88650-7.50011-1},
  url       = {https://www.sciencedirect.com/science/article/pii/B9780444886507500111}
}

@unpublished{verma1993graphical,
  title  = {Graphical Aspects of Causal Models},
  author = {Verma, Thomas},
  note   = {UCLA Cognitive Systems Laboratory, Technical Report (R-191)},
  year   = {1993}
}

@inproceedings{yang2014mixed,
  title     = {Mixed {{Graphical Models}} via {{Exponential Families}}},
  booktitle = {Proceedings of the {{Seventeenth International Conference}} on {{Artificial Intelligence}} and {{Statistics}}},
  author    = {Yang, Eunho and Baker, Yulia and Ravikumar, Pradeep and Allen, Genevera and Liu, Zhandong},
  year      = {2014},
  month     = apr,
  pages     = {1042--1050},
  publisher = {PMLR},
  issn      = {1938-7228},
  url       = {https://proceedings.mlr.press/v33/yang14a.html},
  langid    = {english}
}

@inproceedings{zhang2002strong,
  title     = {Strong Faithfulness and Uniform Consistency in Causal Inference},
  booktitle = {Proceedings of the {{Nineteenth}} Conference on {{Uncertainty}} in {{Artificial Intelligence}}},
  author    = {Zhang, Jiji and Spirtes, Peter},
  year      = {2002},
  month     = aug,
  series    = {{{UAI}}'03},
  pages     = {632--639},
  publisher = {Morgan Kaufmann Publishers Inc.},
  address   = {San Francisco, CA, USA},
  isbn      = {978-0-12-705664-7}
}

@article{zhang2008detection,
  title     = {Detection of {{Unfaithfulness}} and {{Robust Causal Inference}}},
  author    = {Zhang, Jiji and Spirtes, Peter},
  year      = {2008},
  month     = jun,
  journal   = {Minds and Machines},
  volume    = {18},
  number    = {2},
  pages     = {239--271},
  issn      = {0924-6495, 1572-8641},
  doi       = {10.1007/s11023-008-9096-4},
  url       = {http://link.springer.com/10.1007/s11023-008-9096-4},
  copyright = {http://www.springer.com/tdm},
  langid    = {english}
}

@article{zhang2013comparison,
  title     = {A {{Comparison}} of {{Three Occam}}'s {{Razors}} for {{Markovian Causal Models}}},
  author    = {Zhang, Jiji},
  year      = 2013,
  month     = jun,
  journal   = {The British Journal for the Philosophy of Science},
  volume    = {64},
  number    = {2},
  pages     = {423--448},
  publisher = {The University of Chicago Press},
  issn      = {0007-0882},
  doi       = {10.1093/bjps/axs005},
  url       = {https://www.journals.uchicago.edu/doi/abs/10.1093/bjps/axs005}
}


\end{document}